\documentclass[titlepage,12pt]{article}
\usepackage[top=2.5cm,left=2.5cm,right=2.5cm,bottom=2.5cm]{geometry}

\usepackage{amsmath}
\usepackage{amsthm}
\usepackage{amssymb}
\usepackage{amscd}
\usepackage{amsfonts}
\usepackage{amsbsy}
\usepackage{graphicx}
\usepackage[dvips]{psfrag}
\usepackage{setspace}
\usepackage{indentfirst}
\usepackage{rotating,graphics,psfrag,epsfig}
\usepackage{float}
\usepackage[colorlinks=true, linkcolor=blue, citecolor=blue]{hyperref}
\usepackage{booktabs}

\usepackage[symbol]{footmisc} 

\newtheorem{theorem}{Theorem}

\newtheorem{proposition}{Proposition}
\newtheorem{lemma}{Lemma}
\renewenvironment{proof}{\noindent {\bfseries Proof.}}{\hfill $\blacksquare$\vspace{0.4cm}}

\begin{document}


\onehalfspacing

\begin{center}
{\Large{\bf Resilience Analysis for Competing Populations\footnote{Preprint of an article submitted for consideration in \textit{Bulletin of Mathematical Biology}.
}}}
\end{center}

\medskip

\begin{center}
{\large {\large Artur C\' esar Fassoni}} 
\\
{\em Instituto de Matem\'atica e Computa\c c\~ao, Universidade Federal de
Itajub\'a,\\ Avenida BPS 1303, Pinheirinho, CEP 37.500--903, Itajub\'a,
MG, Brazil.\\
}
e--mail: fassoni@unifei.edu.br
\end{center}

\begin{center}
{\large Denis de Carvalho Braga}
\\
{\em Instituto de Matem\'atica e Computa\c c\~ao, Universidade Federal de
Itajub\'a,\\ Avenida BPS 1303, Pinheirinho, CEP 37.500--903, Itajub\'a,
MG, Brazil.\\
}
e--mail: braga@unifei.edu.br
\end{center}

\medskip

\begin{center}
{\bf Abstract}
\end{center}

Ecological resilience refers to the ability of a system to retain its state when subject to state variables perturbations or parameter changes. While understanding and quantifying resilience is crucial to anticipate the possible regime shifts, characterizing the influence of the system parameters on resilience is the first step towards controlling the system to avoid undesirable critical transitions. In this paper, we apply tools of qualitative theory of differential equations to study the resilience of competing populations as modeled by the classical Lotka-Volterra system. Within the high interspecific competition regime, such model exhibits bistability, and the boundary between the basins of attraction corresponding to exclusive survival of each population is the stable manifold of a saddle-point. Studying such manifold and its behavior in terms of the model parameters, we characterized the populations resilience: while increasing competitiveness leads to higher resilience, it is not always the case with respect to reproduction. Within a pioneering context where both populations initiate with few individuals, increasing reproduction leads to an increase in resilience; however, within an environment previously dominated by one population and then invaded by the other, an increase in resilience is obtained by decreasing the reproduction rate. Besides providing interesting insights for the dynamics of competing population, this work brings near to each other the theoretical concepts of ecological resilience and the mathematical methods of differential equations and stimulates the development and application of new mathematical tools for ecological resilience.

\bigskip

\noindent {\small {\bf Key--words}: Nonlinear dynamics, invariant manifolds, basins of attraction, biological invasions, ecological resilience.}

\noindent {\small {\bf 2010 Mathematics Subject Classification:} 37C27, 37G15, 37C75.}

\newpage

\section{Introduction}\label{secIntro}

Ecological resilience refers to the ability of a system to retain its present state under environmental disturbance \cite{walker2004resilience,folke2006resilience}. State variable perturbations beyond a critical threshold drive the system outside the current basin of attraction and lead to a regime shift, while parameter changes caused by environmental change may move the basin thresholds or cause bifurcations. Because real biological systems are constantly subject to such perturbations and changes, understanding the resilience of a system and predicting the occurrence of a regime shift in advance is of major importance and increasing necessity for ecology and biology problems such as biological invasions, climate changing, agricultural management, cancer, cell death, neuronal dynamics, etc \cite{scheffer2001catastrophic, barnosky2012approaching,  trotta2012global,scheffer2013migraine,fassoni2014mathematical,fassoni2014basins,konstorum2016feedback, fassoni2017ecological}.

\smallskip

The study of mathematical models within the ecological resilience framework can be divided into three different aspects or layers. First, one has to {\it define resilience}, i.e., the concepts of ecological resilience, which are summarized with a simple question: ``resilience from what to what?'' \cite{carpenter2001metaphor}. Then, it is necessary to \textit{quantify resilience}, i.e., to develop and use methods which transform the concepts in numerical quantities and data \cite{mitra2015integrative, meyer2016mathematical}. The final step is to actually \textit{apply the concepts and methods} in order to derive qualitative and quantitative conclusions about the system resilience \cite{menck2013basin,fassoni2017ecological}. The resilience approach complements and enlarges the linear-stability paradigm and should result in new insights, predictions and control strategies for the modeled phenomena, providing answers to questions such as \textit{how far is the system near to a tipping point? How to increase a basin of attraction? How to prevent a critical transition? How to move the system to another basin?}

\smallskip

From the theoretical point of view, the importance of resilience was raised by Holling in 1970 \cite{holling1973resilience}, and several aspects within this theory are well defined \cite{carpenter2001metaphor,scheffer2001catastrophic,walker2004resilience,folke2006resilience}. The resilience of a system can be classified with respect two different types of change: resilience to state variable perturbations and resilience to parameter changes. The first concerns temporal disturbances in the variables governed by the intrinsic laws of that system, while the second refers to changes in the intensities of these laws themselves. While state variable perturbations may lead to a change in the present basin of attraction, parameter changes modify the phase portrait, thereby changing the position and shape of basin boundaries, and ultimately lead to bifurcations, implying in the disappearance of a stable state or loss of its stability. In real systems, both types of perturbations occur along time and may be caused by random fluctuations inherent to the system or by external forces. The combination of both perturbations is actually the most likely to drive the system to a critical transition: a state variable perturbation itself may not reach the required basin threshold but a change in some parameter may facilitate this transition by moving the position of such threshold.

\smallskip

At least three different measures are well defined to quantify the resilience of a system with respect to state variable perturbations \cite{walker2004resilience,meyer2016mathematical}. By quantifying distinct aspects of the size and shape of a basin of attraction, these three measures are complementary to each other and together provide a complete scenario for analyzing resilience. The first is called \textit{precariousness} and measures the distance from the actual state to the basin boundary (which may or not be restricted to a set of plausible directions of perturbations). It measures the minimal perturbation needed to drive the system trajectory outside the current basin of attraction. Another measure is the \textit{resistance} which quantifies the basin steepness or the magnitude of the vector field toward the actual state. It measures how fast the system trajectory returns to the initial point after a given perturbation. Finally, the \textit{latitude} measures the whole basin size, for instance, the basin area for two-dimensional systems. It is proportional to the probability that a random perturbation (or initial condition) drives the system trajectory to that basin of attraction.

\smallskip

With respect to the second methodological aspect, although the concepts are well defined, one can say that the methods to quantify ecological resilience in a system of ordinary differential equations are still at an initial stage of development and are specially complex for high-dimensional systems \cite{meyer2016mathematical}. As can be anticipated from the above description, assessing all the three resilience measures (precariousness, resistance and latitude) requires calculating the separatrix between the basins of attraction. For a $n$-dimensional system, the separatrix in general is a $(n-1)$-dimensional manifold formed by stable manifolds of saddle-points or limit cycles \cite{chiang1988stability}. While numerical and analytical methods to estimate invariant manifolds can be easily applied in one or two-dimensional system \cite{genesio1985estimation}, higher dimensions require more complex approaches. Numerical continuation methods are very accurate even when applied to systems exhibiting complex dynamics (such as the Lorenz system) but involve sophisticated implementation \cite{krauskopf2007computing,guckenheimer2004fast}, while simpler methods using basin sampling \cite{cavoretto2016robust, fassoni2014mathematical,delboni2017mathematical} or Monte Carlo approaches \cite{mitra2015integrative} may achieve a good compromise between accuracy and ease of use, especially when the separatrix is described as a graph of a function \cite{cavoretto2016robust}. Analytical methods such as the Picard method and power series method were successfully applied to two dimensional systems and may provide useful analytical information, such as an expression of the separatrix slope or curvature as a function of parameters \cite{fassoni2014basins,konstorum2016feedback, fassoni2017ecological}. The application of these analytic methods may be facilitated by computer algebra systems, but even so these are difficult to devise in models with more than two dimensions. A recent publication introduced a \textit{flow-kick} framework to quantify resilience explicitly in terms of the magnitude and frequency of repeated disturbances \cite{meyer2018quantifying}.

\smallskip

Finally, we have the third step, the application of concepts and methods and derivation of conclusions about the system resilience. Although multistability is a common phenomenon in many mathematical models for biological phenomena, few works analyzed such models within the the framework of ecological resilience \cite{fassoni2014basins,konstorum2016feedback,fassoni2017ecological}. One reason may be that the theoretical concepts in general come from the ecologists community and are still unknown for the modeling community or still not encoded in a concrete mathematical framework, although a recent review provided an good starting point \cite{meyer2016mathematical}. To cite a few works which somehow applied these concepts, we refer to a model for tumor growth with feedback regulation, where Hillen et. al used numeric and analytic estimates of the basin boundary to derive optimal strategies of tumor treatment \cite{konstorum2016feedback}. Ledzewicz used a linear approximation of the basin boundary to define a criteria for an optimal control problem for cancer treatment \cite{ledzewicz2015dynamics}. Also in a model for tumor growth and treatment, Fassoni used the theoretical framework of ecological resilience and critical transitions to describe the process of cancer onset and cure, and performed a ``resilience analysis'' with respect to change in key parameters \cite{fassoni2017ecological}.

\smallskip

With this work, we aim to contribute to the latter two aspects of using ecological resilience in mathematical models, i.e., to provide a step further in the development of methods and to illustrate how this approach may provide interesting (or even surprising) biological conclusions which potentially may be used for delineating strategies of control or management. We consider the classical Lotka-Volterra model of competition between two populations \cite{murray2007mathematical} as a model system. In a previous work, Fassoni studied this model from a resilience point of view, but most conclusions were obtained numerically for a given set of parameter values \cite{fassoni2014basins}. A recent work studied this model focusing on the ratio between the basins areas \cite{chiralt2017quantitative}. Here, we use techniques from the qualitative theory of ordinary differential equations to derive analytical properties of the basin boundary in the whole parameter space. These mathematical properties translate into different biological implications for the resilience of competing populations regarding the measures of precariousness and latitude. Calculating the derivative of the separatrix with respect to the parameters, we show that when a population increases its competitiveness, the separatrix moves in such way the basin of attraction corresponding to the survival of that population becomes larger. With respect to the proliferation rate, we show that, for low initial densities of both populations, it is advantageous for a species to reproduce fast, whereas for high initial densities, it is advantageous to reproduce slow. Using singular perturbation to analyze the limit cases, our results show that a population that is initially abundant and drastically reduces its proliferation cannot be invaded by a competitor. We believe that, besides providing interesting biological insights for competing populations, the techniques used here constitute a toolbox that can be applied to analyze other ordinary differential equation models from the resilience point of view.

\smallskip

The paper is organized as follows: in Section \ref{secMainRes}, the results are stated as a theorem and their biological implications are discussed. In Section \ref{secMet}, the proofs are presented; this section can be omitted by the biologically oriented reader. Finally, in Section \ref{secConclusion} the conclusion is presented.

\section{Results} \label{secMainRes}

We study the resilience of two competing populations in a scenario described by the classical Lotka-Volterra-Gause model. This is a classical model in Mathematical Biology and it is commonly used as a building block for more complex models in different contexts \cite{murray2007mathematical, hofbauer1998evolutionary}. Here, we interpret the model as describing the scenario where one population is native within an environment which is invaded by a second population. This is a particular interpretation which emphasizes the resilience point of view but our results are also valid in a context where both populations are native or both invade an empty environment.

\subsection{Mathematical statement}

Denoting the abundance of the two populations at time $T$ by $N(T)$ and $I(T)$ (native and invader), the model is given by the following system of ordinary differential equations
\begin{equation}\label{sis1dim}
\left\{
{\setlength\arraycolsep{2pt}
\begin{array}{rcl}
\dfrac{dN}{dT}  & = &  r_N N \left(1- \dfrac{N+aI}{K_N} \right), \vspace{0.2cm}\\
\dfrac{dI}{dT}   & = & r_I I   \left(1- \dfrac{I+bN}{K_I} \right),
\end{array}}
\right. 
\end{equation}
where $r_N$, $r_I$, $K_N$, $K_I$, $a$ and $b$ are positive real numbers \cite{murray2007mathematical}.

The model assumes a limited amount of resources for supplying both populations, which compete for these resources. Therefore, the kinetics of both populations is described by a logistic growth (net growth rates $r_N$ and $r_I$, carrying capacities $K_N$ and $K_I$) and interspecific competition parameters $b$ and $a$. We consider a nondimensional version of system \eqref{sis1dim}, obtained by defining the nondimensional variables and parameters
\[
x=\dfrac{N}{K_N}, \ \
y=\dfrac{I}{K_I}, \ \
t=r_N T,  \ \
\alpha = \dfrac{K_I}{K_N} a, \ \
\beta = \dfrac{K_N}{K_I} b, \ \
\delta = \dfrac{r_I}{r_N}.
\]
The nondimensional model is given by
\begin{equation}\label{siscompeticao}
\left\{
{\setlength\arraycolsep{2pt}
\begin{array}{rcl}
\dfrac{dx}{dt}  & = &  x \left(1-x-\alpha y\right), \vspace{0.2cm} \\
\dfrac{dy}{dt}   & = & \delta y   \left(1-y-\beta x\right).
\end{array}}
\right.
\end{equation}

Within the strong competition regime, the competitive exclusion principle is observed: only one population survives. Mathematically, the strong competition regime corresponds to parameter conditions $\alpha , \beta>1$, which means that the competitive pressure that one population exerts on the other population (interspecific competition) is higher than the pressure that the population exert on itself (intraspecific competition). In this scenario, the equilibrium points corresponding to survival of one population and extinction of the other, $P_N=(1,0)$ and $P_I=(0,1)$, are locally asymptotically stable and the phase space is divided into two regions, the basins of attraction of $P_N$ and $P_I$ (see Figure \ref{FigMain}). The separatrix between these regions is the stable manifold $S$ of the saddle point $P_C=\left(A,B\right)$, where $A=(\alpha-1)/(\alpha\beta-1)$ and $B=(\beta - 1)/(\alpha\beta - 1)$. The left branch of $S$ is a heteroclinic connection from equilibrium $P_0=(0,0)$, a unstable node, to $P_C$. In this bistable regime, the outcome (which population will survive) strongly depends on the initial conditions.

\begin{figure}[!htb]
\begin{center}
\includegraphics[width=0.75\linewidth]{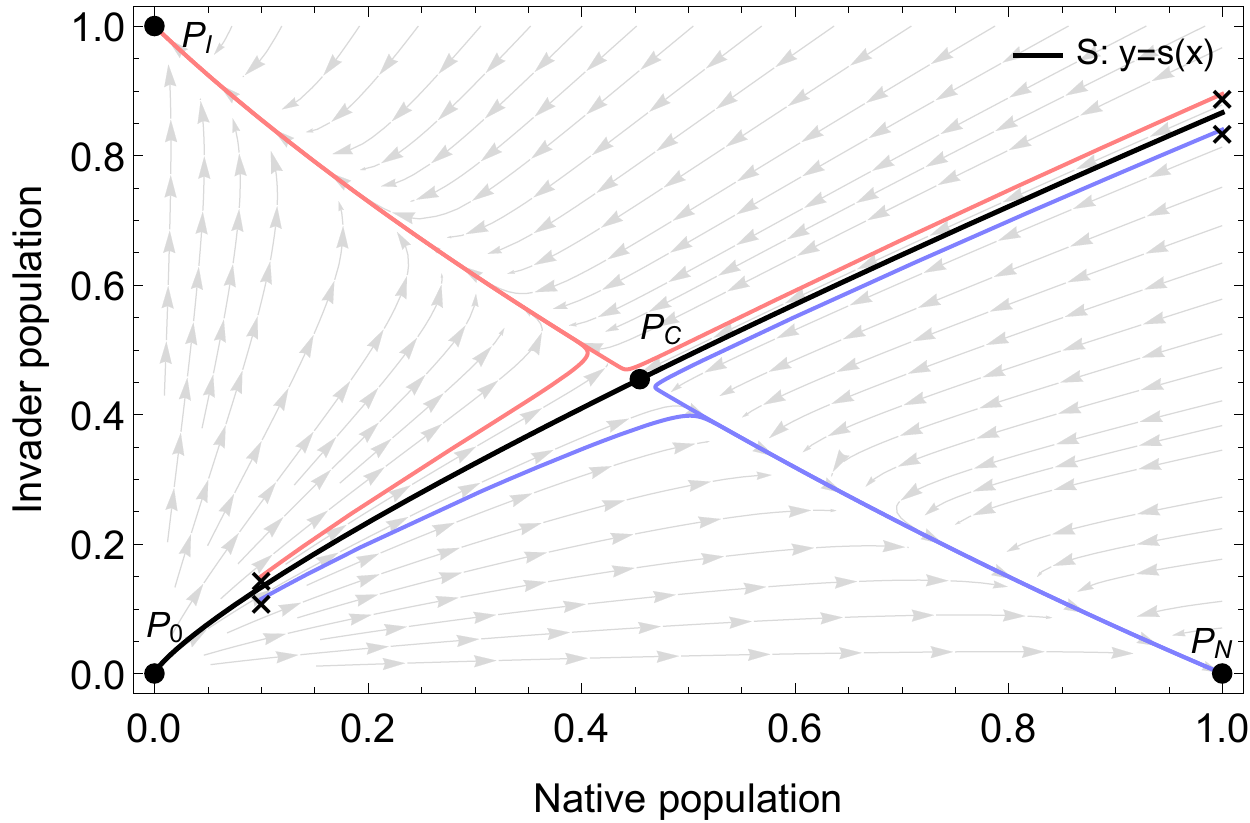}
\caption{Phase portrait of system \eqref{siscompeticao} in the strong competition regime, with $\alpha>1$, $\beta>1$ and $\delta>0$. The equilibrium points $P_N$ and $P_I$ are stable nodes, while $P_C$ is a saddle-point and $P_0$ is an unstable node.The separatrix $S$ between the basins of attraction (black curve) is the stable manifold of $P_C$ with the point $P_0$. Solutions starting below $S$ (blue curves, initial conditions indicated with ``$\times$'') converge to $P_N$, while solutions starting above $S$ (red curves, initial conditions indicated with ``$\times$'') converge to $P_I$. Theorem \ref{teoMain} states that $S$ is a graph of a smooth and increasing function $y=s(x)$.
\label{FigMain}}
\end{center}
\end{figure}

In a previous work, Fassoni used numerical simulations to describe the behavior of the separatrix as the model parameters vary \cite{fassoni2014basins}. Here, we extend those results by analytically characterizing such behavior within the entire parameter space. We also study the limiting cases when time scales of reproduction of each population become very different, i.e., the cases $\delta=r_I/r_N \to 0$ and $1/\delta=r_N/r_I \to 0$. Such results are suggested by the study of a formula which approximates the separatrix and serves as a model for understanding its behavior. Theorem \ref{teoMain} summarizes our results. The proof is presented in the Section \ref{secMet}.

\begin{theorem}\label{teoMain}
Consider system \eqref{siscompeticao} with conditions $\alpha,\beta>1$ and $\delta>0$. Let $S$ be the stable manifold of the saddle point $P_C=(A,B)$, and $\hat{S}=S\cup \{P_0\}$. Then:
\begin{description}
\item[1.] $S$ and $\hat{S}$ have the following properties:
\begin{enumerate}
\item[a.] $S$ is the boundary (or separatrix) between the basins of attraction of $P_N$ and $P_I$;
\item[b.] For fixed parameter values, $(\alpha , \beta , \delta )$, $\hat{S}$ is described as a graph of a smooth function, $x\mapsto y=s(x)$, $x\geq 0$, with $s(0)=0$; 
\item[c.] Solutions with initial conditions $(x_0,y_0)$, $x_0>0,y_0>0$, starting below $S$, i.e., $y_0<s(x_0)$, converge to $P_N$, and solutions starting above $S$, i.e., $y_0>s(x_0)$, converge to $P_I$;
\item[d.] The function $y=s(x)$ is strictly increasing with respect to $x$.
\end{enumerate}
\item[2.] The function $s(x)$ satisfies the integral equation
\begin{equation}
\label{eqIntegralS}
s(x)=s^{\ast}(x) \exp \left ( \delta \int_A^x  \dfrac{(\beta-1) s(\bar{x}) -(\alpha -1) \bar{x}}{\bar{x}(1-\bar{x}-\alpha s(\bar{x}))}\,\mathrm{d}\bar{x} \right).
\end{equation}
Moreover, the function $s^{\ast}(x)=B (x/A)^{\delta}$, $x\geq 0$, is a uniform approximation for $s(x)$ and coincides with $s(x)$ when $\delta=1$.
\item[3.] For each fixed $x>0$:
\begin{enumerate}
\item[a.] The function $\alpha \mapsto s(x,\alpha)$ is strictly decreasing;
\item[b.] The function $\beta \mapsto s(x,\beta)$ is strictly increasing;
\item[c.] The function $\delta \mapsto s(x,\delta)$ is strictly increasing if $x<A$ and is strictly decreasing if $x>A$.
\end{enumerate}
\item[4.] The behavior of $\hat{S}$ for the limit cases $\delta \to 0$ and $1/\delta \to 0$ is the following:
\begin{enumerate}
\item[a.] When $\delta \to 0$, $\hat{S}$ converges to the set given by the union of point $P_0$ with the horizontal line $y=B$, $x>0$, i.e.,
\[
\hat{S} \to S_0=\{P_0\} \cup \{ (x,B), \; x>0 \}, \ \ \ {\rm when} \ \ \  \delta \to 0;
\]
\item[b.] When $1/\delta \to 0$, $\hat{S}$ converges to the set given by the union of point $P_0$ with the vertical line $x=A$, $y>0$, i.e.,
\[
\hat{S} \to S_\infty=\{P_0\} \cup \{ (A,y), \; y>0 \}, \ \ \ {\rm when} \ \ \  1/\delta \to 0.
\]
\end{enumerate}
\end{description}
\end{theorem}

Next, we present the biological interpretation and implications of Theorem \ref{teoMain} from the ecological resilience perspective, focusing on the measures of \textit{precariousness} and \textit{latitude}.

\subsection{Quantifying resilience}

Theorem \ref{teoMain} provides a way to quantify two resilience measures, the \textit{precariousness} and the \textit{latitude}, for the native population. The \textit{precariousness} of an initial state inside a given basin of attraction is a resilience measure quantifying the minimal perturbation required to drive the system trajectory starting at that state to another basin of attraction \cite{walker2004resilience,fassoni2017ecological}. In our context, the precariousness of an initial density of native individuals $x_0 \in [0,1]$, denoted as $Pr(x_0)$, is the minimal density $y_{\min}$ of invader individuals which is able to successfully invade the environment, i.e., to survive and completely eliminate the native population. Mathematically, the precariousness $Pr(x_0)=y_{\min}$ is the minimum value such that the initial condition $(x_0,y_{\min})$ belongs to the basin of attraction of $P_I$. According to Theorem \ref{teoMain}, item 1, and the above definition, we conclude that the precariousness $Pr(x_0)$ is given by the value of the function $s(x)$ whose graph defines the separatrix $S$, evaluated at $x_0$, i.e.,
\[
Pr(x_0)=s(x_0).
\]
From Theorem \ref{teoMain}, item 1 d, we see that the precariousness $Pr(x_0)$ is an increasing function of $x_0$, i.e., the higher is the initial density $x_0$ of native individuals, the higher is the density of invading individuals $s(x_0)$ required to lead to a successful invasion.

The \textit{latitude} of a stable equilibrium point is a resilience measure quantifying the area or volume of its basin of attraction, restricted to a larger set of all eventual initial conditions that might be set by random state-space perturbations \cite{walker2004resilience,fassoni2017ecological}. Thus, the latitude measures the probability that a random initial condition taken within a predefined set of plausible initial conditions is inside the desired basin of attraction. In our context, the desired basin of attraction is ${\cal B} (P_N)$, corresponding to the survival of the native population. The predefined set will be considered as $K=[0,1]\times [0,1]$, meaning that the possible initial conditions $(x_0,y_0) \in K$ correspond to initial densities smaller than the maximum density of each population, $0\leq x_0, y_0 \leq 1$. Thus, the latitude of $P_N$ is the ratio between the areas of ${\cal B} (P_N) \cap K$ and $ K$. Since ${\rm area } (K)=1$, we conclude from Theorem \ref{teoMain}, item 1, that
\[
L(P_N) = {\rm area } ({\cal B} (P_N) \cap K) = \int_0^1 s(x) dx.
\]

In summary, the conclusion given in Theorem \ref{teoMain}, item 1, that the separatrix $S$ is a graph of a function $y=s(x)$, provides an useful way to express the resilience measures \textit{precariousness} and \textit{latitude} in terms of the function $s(x)$. Next, we will illustrate how changes in the model parameters lead to changes in $s(x)$ and, therefore, in these resilience measures.

\subsection{A model function for the separatrix}

Theorem \ref{teoMain}, item 2, provides an approximation $s^*(x)$ for the separatrix equation $y=s(x)$. In Figure \ref{FigApp}, we see that the approximation is very accurate for different combinations of parameter values. More important than the approximation quality itself is the fact that the expression of $s^*(x)$ provides insights on the behavior of $s(x)$ with respect to the parameters and, thus, allows to speculate about their influence on the resilience measures quantified above.

\begin{figure}[!htb]
\begin{center}
\includegraphics[width=0.32\linewidth]{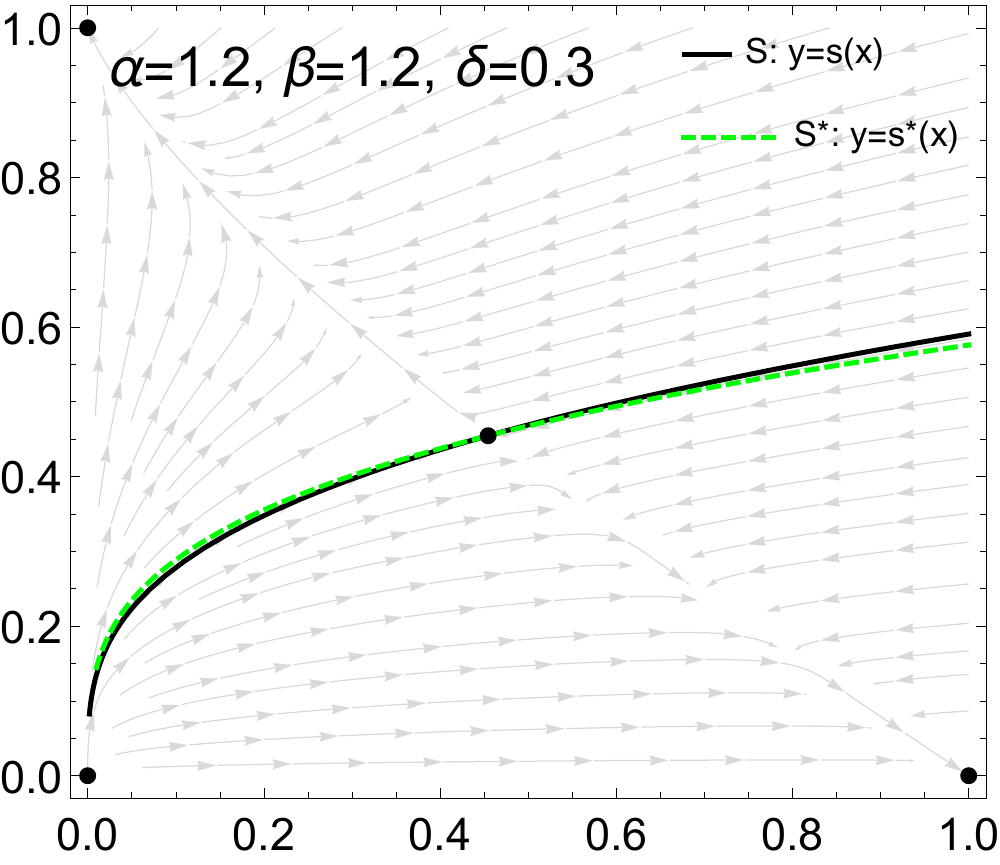}
\includegraphics[width=0.32\linewidth]{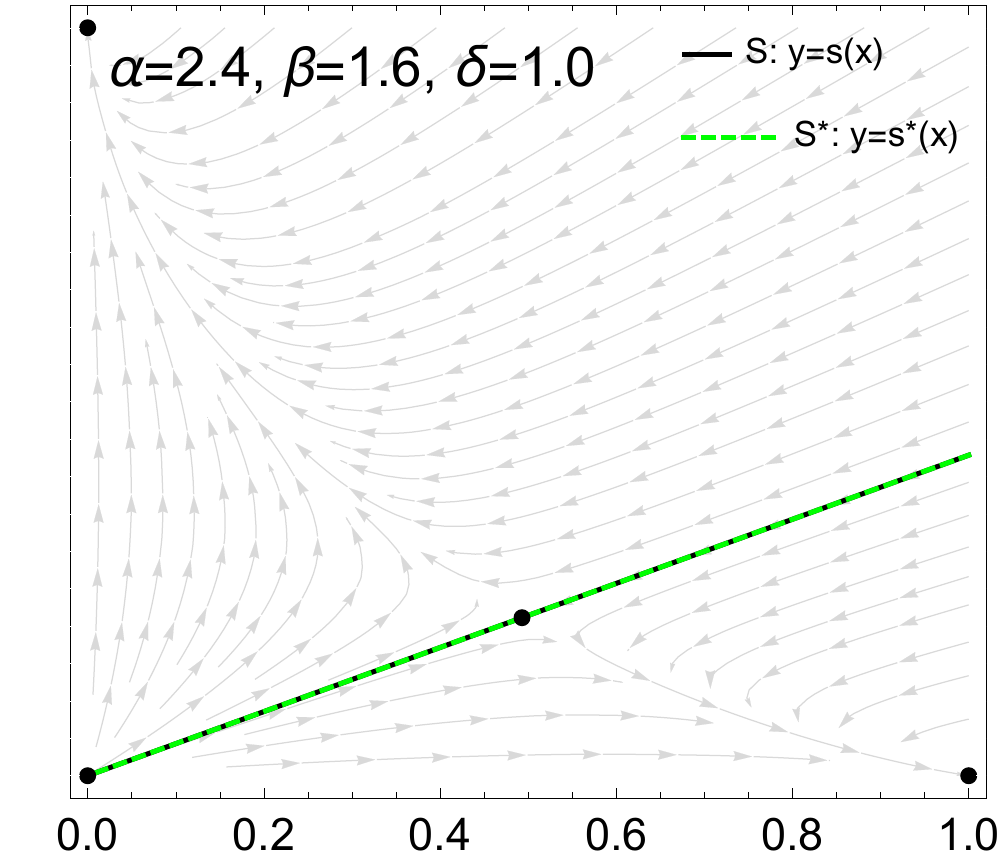}
\includegraphics[width=0.32\linewidth]{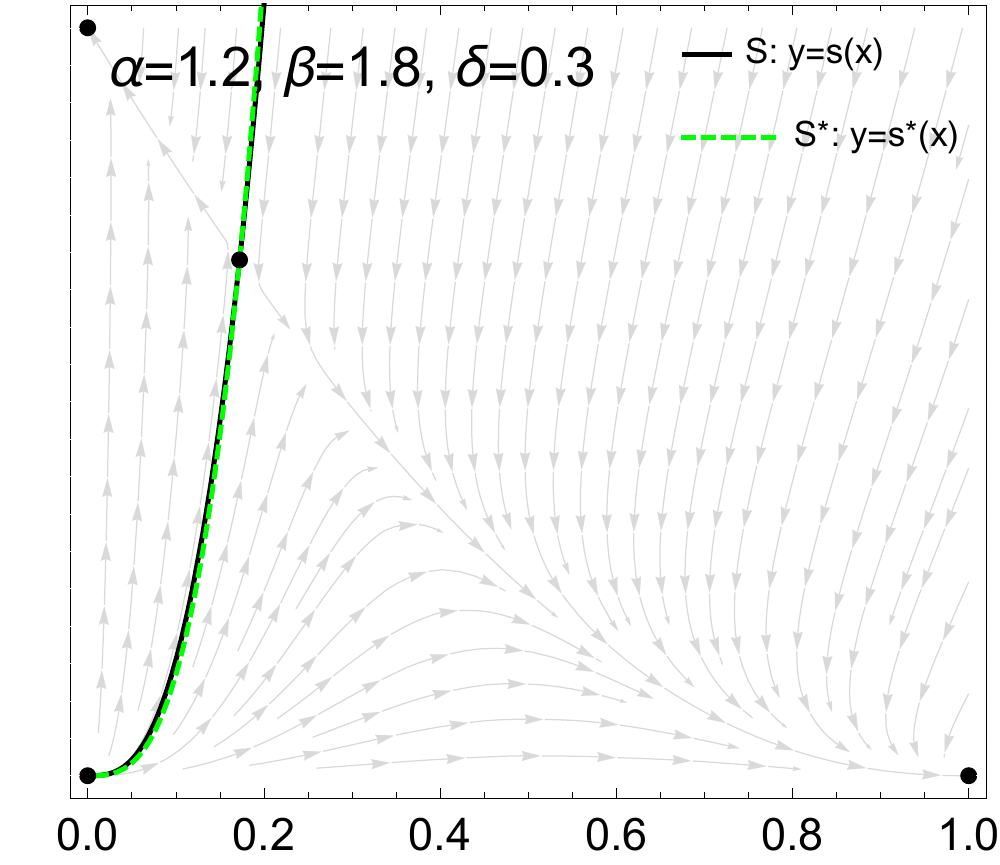}
\caption{Phase portrait of system \eqref{siscompeticao} for different parameter values. The equilibrium points are indicated by black dots. The separatrix $S$, which is the graph of $s(x)$ (black line), is approximated by $s^*(x)=B(x/A)^\delta$ (green line), which also passes trough the point $P_C=(A,B)$. For $\delta=1$ (middle panel), $s(x)$ and $s^*(x)$ coincide and are a straight line.
\label{FigApp}}
\end{center}
\end{figure}

With respect to parameter $\delta$, we note that the function $s^*(x)=B(x/A)^\delta $ is a power in $x$ with exponent $\delta$ and coefficient $B/A^\delta$ and its graph passes through the point $P_C=(A,B)$. Thus, the parameter $\delta$ plays a role similar to the exponent $d$ in the monomial $x^d$. For small values, $0< \delta <1$, the graph of $s^*(x)$ is concave-down and similar to the graph of a $n$-th root. For $\delta=1$, $s^*(x)$ represents a straight line. For high values, $\delta >1$, the graph of $s^*(x)$ is concave-up and, the higher is $\delta$, the smaller is $s^*(x)$ for $x<A$ and the higher is $s^*(x)$ for $x>A$.

To investigate the behavior of $s^*(x)$ with respect the parameters $\alpha$ and $\beta$, we observe that $A=(\alpha-1)/(\alpha\beta-1)$ and $B=(\beta-1)/(\alpha\beta-1)$. Differentiating these expressions with respect to $\alpha$ and $\beta$, we assess the influence of model parameters on the position of the point $(A,B)$ and therefore on the position of $s^*(x)$. We have
\begin{equation}
\dfrac{\partial}{\partial \alpha} \left( A \right) = \dfrac{\beta-1}{(\alpha\beta-1)^2}>0, \ \ \
\dfrac{\partial}{\partial \alpha} \left( B \right) = - \dfrac{\beta(\beta-1)}{(\alpha\beta-1)^2}<0
\label{eqDerivAparam}
\end{equation}
and
\begin{equation}
\dfrac{\partial}{\partial \beta} \left( A \right) = -\dfrac{\alpha(\alpha-1)}{(\alpha\beta-1)^2}<0, \ \ \
\dfrac{\partial}{\partial \beta} \left( B \right) = \dfrac{\alpha-1}{(\alpha\beta-1)^2}>0.
\label{eqDerivBparam}
\end{equation}
Thus, when parameter $\alpha$ increases, the point $P_C=(A,B)$ moves to the right and down, so that we expect that $s^*(x)$ will be smaller for all $x>0$. Analogously, when parameter $\beta$ increases, the point $P_C=(A,B)$ moves to the left and up, so that it is expected that, for all $x>0$, $s^*(x)$ will increase in comparison with the previous value of $\beta$.

In summary, the observed behavior of $s^*(x)$ with respect to the parameters, as described above, raises the question whether $s(x)$ has the same properties. Item 3 of Theorem \ref{teoMain} states that these properties are shared by $s(x)$ and we will see in the following that this fact has implications on the behavior of the resilience measures with respect to parameters changes. Therefore, the approximation $s^*(x)$ provides a model function for the separatrix which allows us to infer the behavior of $s(x)$ and of the resilience measures quantified above.

\subsection{Effects of parameters on resilience}

Theorem \ref{teoMain}, item 3, describes the behavior of the separatrix $S$ with respect to changes in the parameters. Such dependence is emphasized by writing $y=s(x,\xi)$, where $x\geq 0$ and $\xi\in \{\alpha,\beta,\delta\}$ is one of the model parameters.

\begin{figure}[!htb]
\begin{center}
\includegraphics[width=0.49\linewidth]{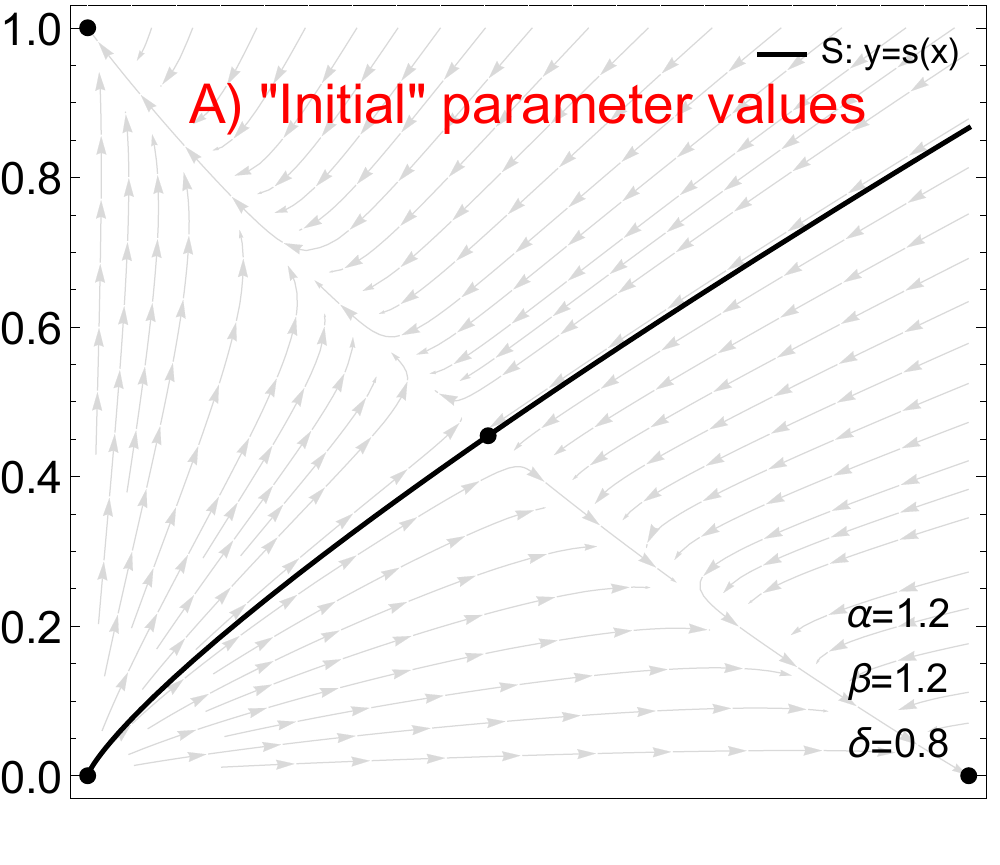}
\includegraphics[width=0.49\linewidth]{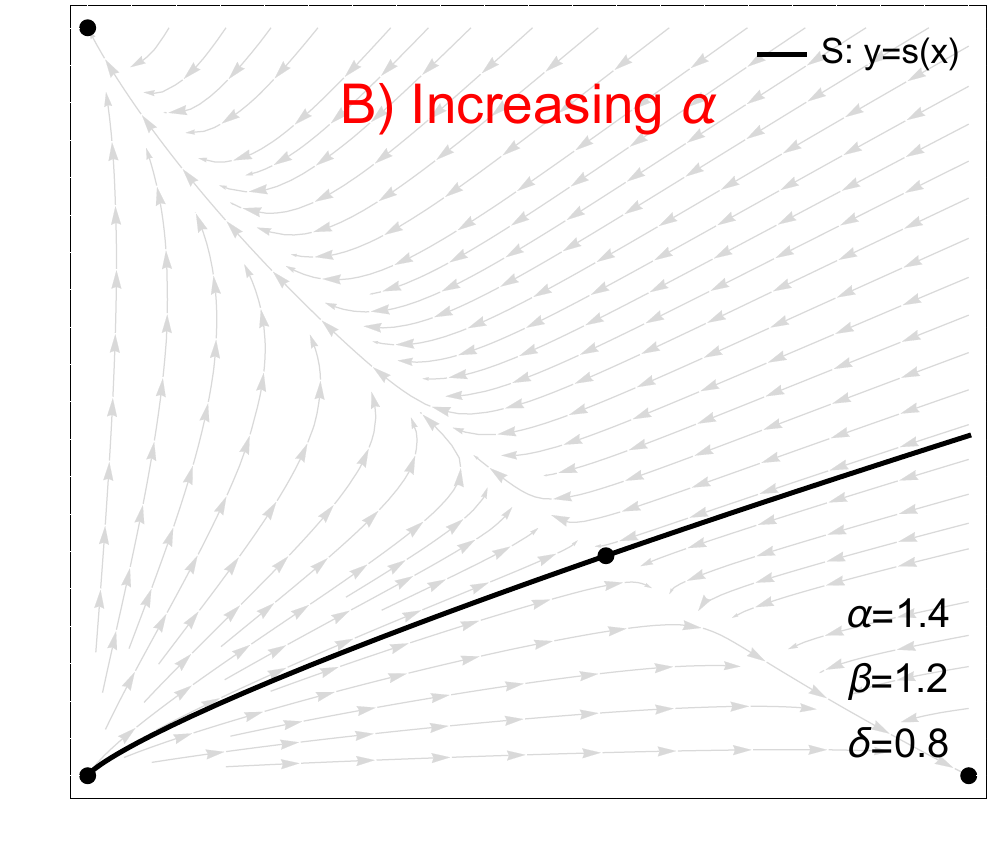}
\includegraphics[width=0.49\linewidth]{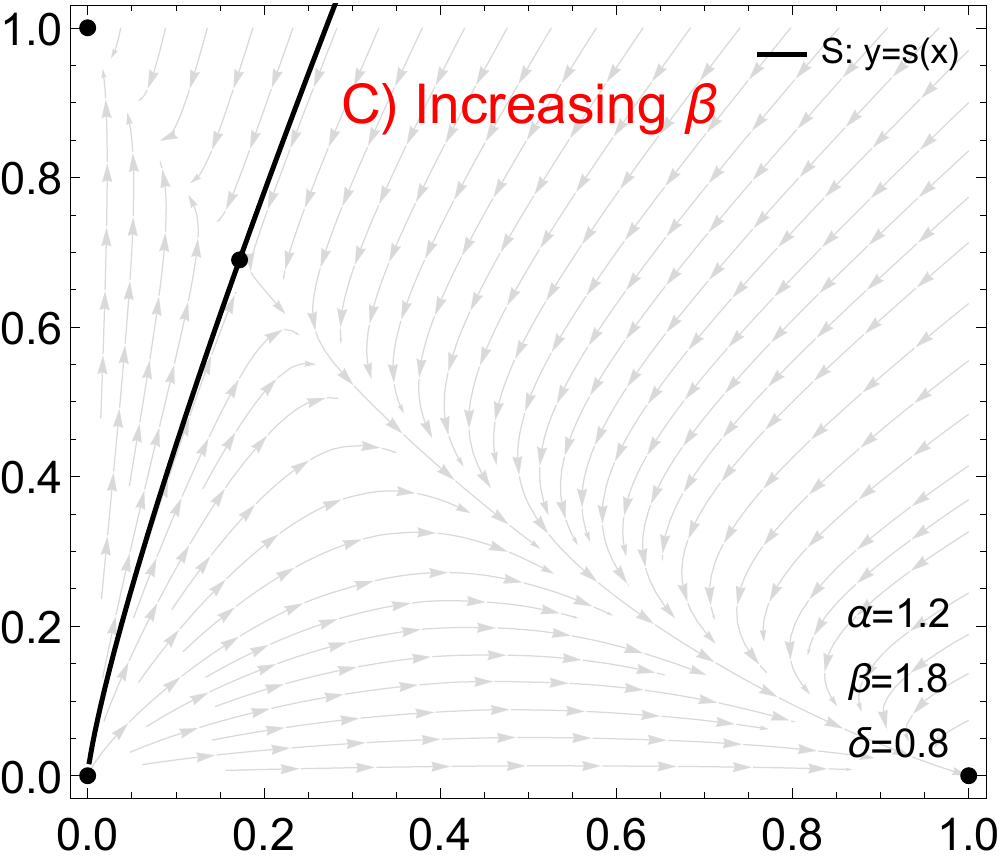}
\includegraphics[width=0.49\linewidth]{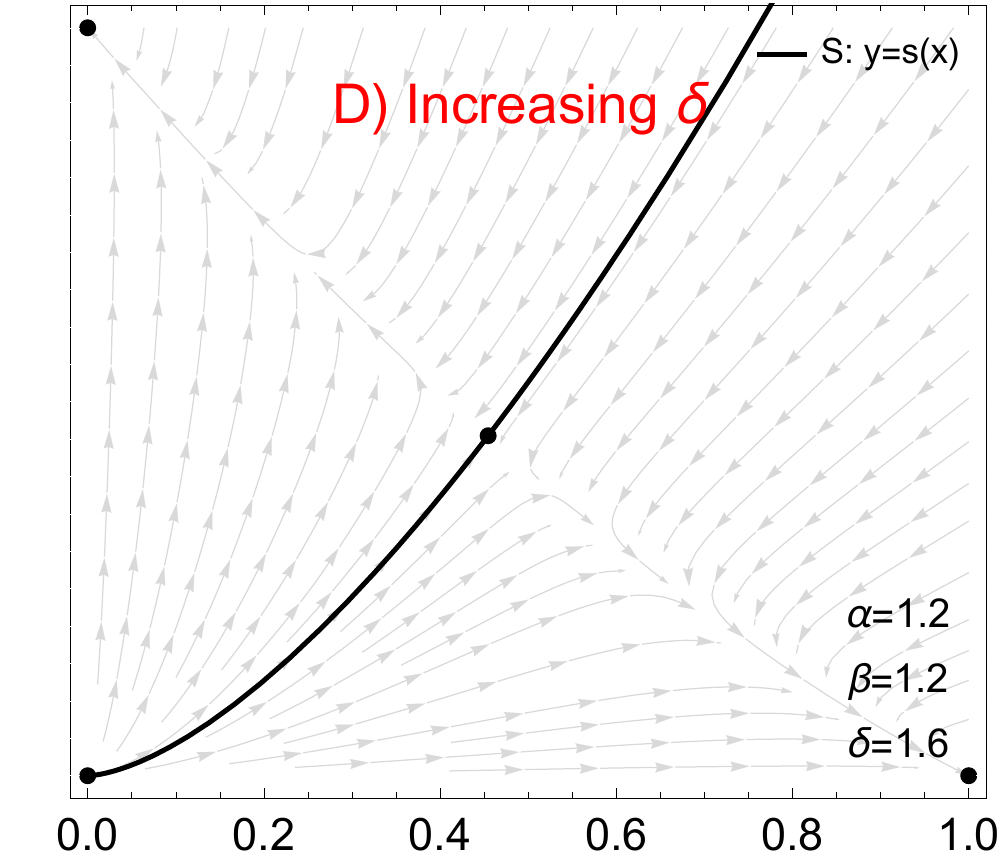}
\caption{Phase portraits of system \eqref{siscompeticao} for different values of parameter. Panel A uses a ``standard'' set of parameter values and each panel illustrate the behavior of the separatrix $S$ as one of the parameter increases while the other remain with the standard values. When the competitiveness $\alpha$ of the invader population increases, the separatrix $S$ moves down, thereby reducing the threshold for invasion and the area of the basin of attraction of $P_N$, which corresponds to the survival of the native population (compare panels A and B). Similarly, when the competitiveness $\beta$ of the native population increases, the separatrix $S$ moves up, thereby increasing the threshold for invasion and the area of the basin of attraction of $P_N$ (compare panels A and C). Finally, when the ratio $\delta=r_I/r_N$ between the reproduction rates increases, the shape of the separatrix changes, similarly to the behavior of the curve $y=x^d$ when the exponent $d$ increases (compare panels A and D); note that the position of $P_C=(A,B)$ does not depends on $\delta$.
\label{FigParamRes}}
\end{center}
\end{figure}

First, we see that $s(x,\alpha)$ is a decreasing function with respect to $\alpha$, for each fixed $x>0$. It means that the separatrix moves down when $\alpha$ increases (Figure \ref{FigParamRes}, panels A and B). This has immediate implications on the resilience measures for the native population. Remember that $\alpha$ represents the competitiveness of the invader population against the native. Also, the precariousness of a given initial number of native individuals is $Pr(x_0)=s(x_0,\alpha)$ and the latitude is given by $\int_0^1s(x,\alpha)dx$. Therefore, we conclude that when the competitiveness $\alpha$ of the invader population increases, both resilience measures (precariousness and latitude) of the native decreases, i.e., the more competitive is the invader population, the lower are the number of invader individuals required to lead to a successful invasion and the higher is the probability that a random initial condition leads to the extinction of the native population. Similarly, since $s(x,\beta)$ is increasing with respect to $\beta$, for fixed $x>0$, it follows that the higher is the competitiveness $\beta$ of the native population, the higher are its resilience measures against the invader population. Therefore, we conclude that higher competitiveness favors resilience, which is already expected.

Now, assessing the behavior of $S$ with respect to $\delta$, we obtain interesting conclusions. For each $x_0<A$ fixed, $s(x_0,\delta)$ is decreasing with respect to $\delta$. Since  $\delta= r_I/r_N$, it means that the precariousness $Pr(x_0)=s(x_0,\delta)$ of an initial density of native individuals decreases when the reproduction rate $r_I $ of the invader population increases or when the reproduction rate $r_N$ of the native population decreases. On the contrary, the precariousness $Pr(x_0)=s(x_0,\delta)$ increases when $r_I$ decreases or $r_N$ increases. In other words, the higher is the reproduction rate of the native population, the more protected the native population against an invasion. Or, in order to increase the chances of a successful invasion, the invader population should increase its reproduction rate. In summary, within a \textit{pioneering} context, where the initial densities of both population are small ($0<x_0<A$, $0<y_0<B$), it is advantageous to each population to increase its reproduction rate in order to survive the competition with the other population.

In the other case, i.e., for high initial densities $x_0>A$, the behavior is opposite. The value of $Pr(x_0)=s(x_0,\delta)$ is increasing with respect to $\delta= r_I/r_N$. Therefore, if the native population increases its reproduction rate, it decreases its protection against an invasion, since the number of invader individuals required to a successful invasion is smaller. From the perspective of the invader population, the invasion may be facilitated if it decreases its reproduction rate $r_I$. Thus, within a \textit{invasion} context, where the initial density of the native population is high, it is advantageous to each population to decrease its reproduction rate in order to survive. This counterintuitive conclusion may be explained as follows. Besides the interspecific competition against the other population, each population has a intraspecific competition between its individuals. Thus, within a context of high densities and high interspecific competition, the intraspecific competition is an additional limiting factor for the population survival and may be reduced if the population decreases its reproduction rate, leading to a higher chance of survival for the initial density of such population.

\subsection{The basin boundary in the extreme cases}

Theorem \ref{teoMain}, item 4, regards the behavior of $S$ as a function of parameter $\delta$ in the extreme cases, i.e., when $\delta\to 0$ and $1/\delta \to 0$, thereby allowing us to extend the above conclusions to these cases. See Figure \ref{FigAsymp}.

\begin{figure}[!htb]
\begin{center}
\includegraphics[width=0.49\linewidth]{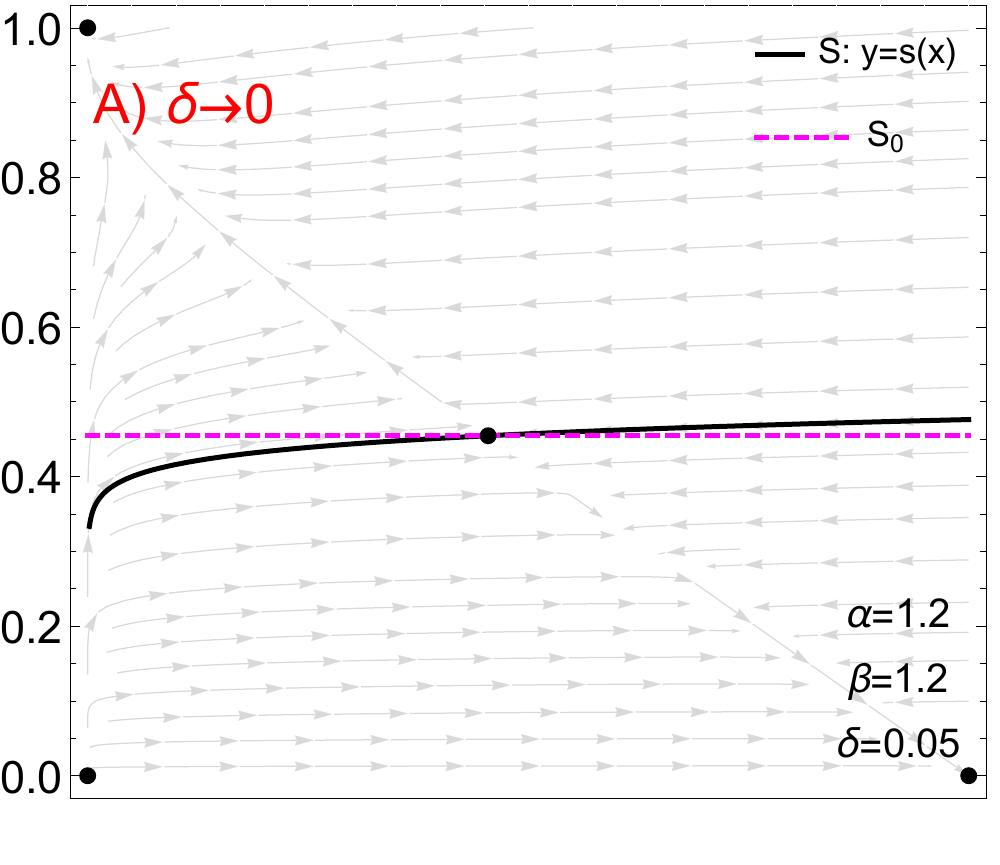}
\includegraphics[width=0.49\linewidth]{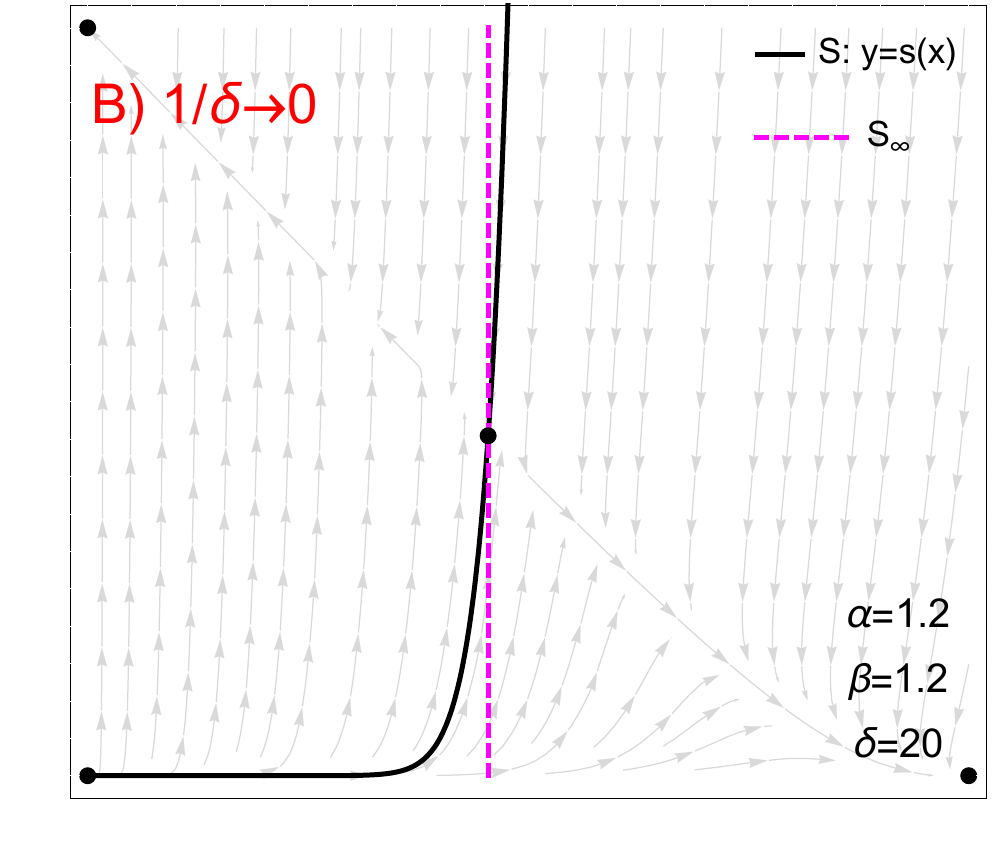}
\caption{Phase portrait of system \eqref{siscompeticao} in the extreme cases, A) $\delta \to 0$ and B) $1/\delta \to 0$ (right). A) When $\delta \to 0$, the separatrix $S$ converges to $S_0$, which is the horizontal line $y=B$ (dashed purple). B) When $1/\delta \to 0$, the separatrix $S$ converges to $S_\infty$, which is the vertical line $x=A$ (dashed purple).
\label{FigAsymp}}
\end{center}
\end{figure}

The case $\delta\to 0$ corresponds to $r_I/r_N \to 0$. It means that the reproduction of the invader population occurs on a much slower time scale in comparison with the reproduction of the native population. When it happens, the separatrix $S$ approaches the horizontal line $y=B$ and so the precariousness $Pr(x_0)$ becomes constant, $Pr(x_0)=B$. In the invasion context, i.e., for $x_0>A$, $Pr(x_0)$ is the minimum value possible among the others when $\delta$ varies. This means that the invader population maximizes it chances of successful invasion when it drastically reduces its reproduction rate in comparison with the native population. It can also be understood from the native point of view: if the native population drastically increases its reproduction rate in comparison with the invader, then it has the least chances to survive the invasion. In the pioneering context, i.e., $x_0<A$, the result is opposite. A very fast reproduction of the native population in comparison to the invader leads to the highest chance for its establishment within the environment.

Finally, the case $1/ \delta = r_N/r_I \to 0$ means that the native population reproduces on a much slower time scale in comparison with the invader population. In this case and within the invasion context, the precariousness $P_r(x_0)$ becomes infinity. This leads to the interesting conclusion that the native population cannot be invaded, whatever be the initial density of invaders, provided that the native reproduces on a much slower time scale and has a sufficient high initial density, $x_0>A$. In other words, during a single generation of the native population, several generations of the invader pass; these are subject to the competitiveness with both the native and its own individuals and are not able to survive and eliminate the native individuals during their short lifetime.

\section{Methods}
\label{secMet}

In this section, we demonstrate Theorem \ref{teoMain}. We start by introducing the notation and preliminary results. The vector field associated with system \eqref{siscompeticao} will be denoted as
\[
F(x,y)=\left(f(x,y), g(x,y) \right) = \left( x\left(1-x-\alpha y\right) , \delta y   \left(1-y-\beta x\right) \right).
\]
The biologically relevant phase space is restricted to the first quadrant $\mathbb{R}^2_+=\lbrace (x,y) : x \geq 0, y\geq 0\rbrace$. For each point $(x,y)\in \mathbb{R}^2_+$, the solution of system \eqref{siscompeticao} starting at $(x,y)$ will be denoted as $\phi(t,x,y)$. We say that a function $x\mapsto y(x)$ is smooth if it is real analytic.

We start with two propositions stating the configuration of the phase portrait of system \eqref{siscompeticao} in the bistable regime (see Figure \ref{FigMain}). The proofs of these results are straightforward and were presented in several previous references \cite{hofbauer1998evolutionary,fassoni2014basins}. For sake of completeness, we present it here, but in a condensed way.
\begin{proposition}
\label{propIntroEq}
Under conditions $\alpha,\beta>1$ and $\delta>0$, the equilibrium points of system \eqref{siscompeticao} are:
\begin{description}
\item[1.] The trivial equilibrium point $P_0=(0,0)$, which is an unstable node;
\item[2.] The ``native-wins'' equilibrium point $P_N=(1,0)$, which  is a stable node;
\item[3.] The ``invader-wins'' equilibrium point $P_I=(0,1)$, which  is a stable node;
\item[4.] The coexistence equilibrium point $P_C=(A,B)$, which  is a saddle-point, where $A=(\alpha-1)/(\alpha\beta-1)$ and $B=(\beta - 1)/(\alpha\beta - 1)$.
\end{description}
\end{proposition}
\begin{proof}
Solving the nonlinear system $F(x,y)=(0,0)$ we find the four equilibrium points $P_0$, $P_N$, $P_I$ and $P_C$. The conditions $\alpha>1$ and $\beta>1$ imply that $A>0$ and $B>0$. Calculating the Jacobian matrix of $F$ at each of these points, the stability is easily assessed for the first three equilibria. The stability of $P_C$ is assessed by analyzing the signs of the determinant and trace of the Jacobian matrix evaluated at $P_C$.
\end{proof}

\begin{proposition}
\label{propIntroSols}
Under conditions $\alpha,\beta>1$ and $\delta>0$, the phase portrait of system \eqref{siscompeticao} \emph{(}restricted to the phase space $\mathbb{R}^2_+$\emph{)} has the following properties:
\begin{description}
\item[1.] System \eqref{siscompeticao} has no limit cycles;
\item[2.] All $\omega$-limit sets of system \eqref{siscompeticao} are contained in the box $K=[0,1]\times [0,1]$, i.e., all trajectories satisfy $x(t)\leq 1$, $y(t)\leq 1$ for $t\to \infty$;
\item[3.] All trajectories of system \eqref{siscompeticao} converge to one of the equilibrium points, i.e., the phase space $\mathbb{R}^2_+$ is decomposed as the following disjoint union,
\[
\mathbb{R}^2_+ = \lbrace P_0 \rbrace \cup S \cup {\cal B}_N \cup {\cal B}_I,
\]
where ${\cal B}_N $ and $ {\cal B}_I$ are the basins of attraction of $P_N$ and $P_I$, respectively, and $S$ is the one-dimensional stable manifold of $P_C$.
\end{description}
\end{proposition}
\begin{proof}
\begin{description}
\item[1.] Setting $u(x,y)=1/xy$, we have ${\rm div}\,(u(x,y) F(x,y))<0$, which implies by the Dulac Criterion that system \eqref{siscompeticao} admits no limit cycles;
\item[2.] Noting that $dx/dt<x(1-x)$ and $dy/dt<\delta y(1-y)$ and using standard comparison principles, we obtain that all solutions of system \eqref{siscompeticao} are bounded and satisfy $x(t)\leq 1$, $y(t)\leq 1$ for $t\to \infty$;
\item[3.] From the first two facts, we conclude by the Poincar\'e--Bendixson theorem that all trajectories converge to equilibrium points. Since $P_C$ is a saddle point of a two dimensional system, its stable manifold is a one-dimensional manifold.
\end{description}
\end{proof}

We now present the proof of Theorem \ref{teoMain}, item per item.

\subsection{Proof of theorem \ref{teoMain}, item 1}
The following two lemmas will be important for characterizing the behavior of $S$ as a graph of a function $y=s(x)$. See Figure \ref{FigLemas} for an illustration of the content of such lemmas.

\begin{figure}[!htb]
\begin{center}
\includegraphics[width=0.8\linewidth]{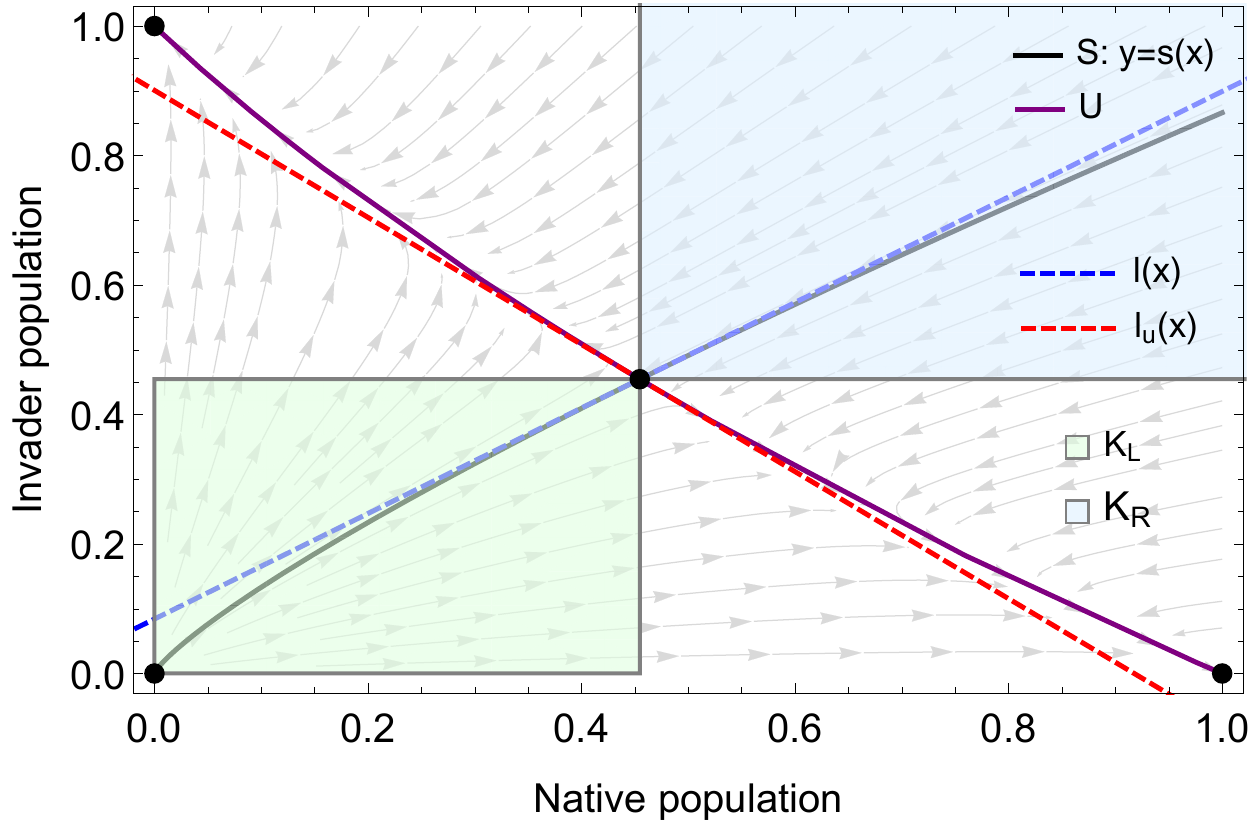}
\caption{
Illustration of Lemmas \ref{LemaReta} and \ref{LemaKs}. The stable and unstable manifolds of the saddle point $P_C=(A,B)$ are, respectively, the curves $S$ (black) and $U$ (purple). The tangent line to the stable manifold $S$ at $P_C$ is $l(x)$ (dashed blue) and tangent line to the unstable manifold $U$ at $P_C$ is $l_u(x)$ (dashed red). The sets $K_L$ (green) and $K_R$ (blue) are negatively invariant with respect to the flow of the vector field $F$ (note the direction of the vector field, indicated by the gray arrows, along the boundaries of $K_L$ and $K_R$).
\label{FigLemas}}
\end{center}
\end{figure}

\begin{lemma}
\label{LemaReta}
Assume $\alpha,\beta>1$ and $\delta>0$. Denote by $S$ and $U$ the stable and unstable manifolds of the saddle point $P_C=(A,B)$. The tangent line to the stable manifold $S$ at $P_C$ is described by the equation
\begin{equation}
y=l(x) =B+m\left(x-A\right),
\label{eqTangLine}
\end{equation}
where the positive slope $m$ is given by
\begin{equation}
m=\dfrac{\left(\sqrt{\Delta }-\eta \right)}{2\alpha (\alpha -1)},
\label{eqSlopem}
\end{equation}
with
\begin{equation}
\Delta = \kappa^2 + 4\delta (\alpha -1)(\beta -1) (\alpha  \beta -1),
\ \
\kappa =(\alpha -1)+\delta (\beta -1),
\ \
\eta =(\alpha -1)-\delta (\beta -1).
\end{equation}
Similarly, the tangent line to the unstable manifold $U$ at $P_C$ is described by the equation
\begin{equation}
y=l_u(x) =B+m_u\left(x-A\right),
\label{eqTangLineU}
\end{equation}
where the negative slope $m_u$ is given by
\begin{equation}
m_u=-\dfrac{\left(\sqrt{\Delta }+\eta \right)}{2\alpha (\alpha -1)}.
\label{eqSlopemU}
\end{equation}
\end{lemma}

\begin{proof}
Calculating the Jacobian matrix of $F(x,y)$ and evaluating it at $P_C$, we find that its eigenvalues are
\[
\lambda_1=\dfrac{-\kappa-\sqrt{\Delta}}{2(\alpha\beta-1)}
\ \ \ \
{\rm and}
\ \ \ \
\lambda_2=\dfrac{-\kappa+\sqrt{\Delta}}{2(\alpha\beta-1)}.
\]
The conditions $\alpha,\beta>1$ and $\delta>0$ imply that $\kappa>0$ and $\Delta>0$. Therefore, $\lambda_1<0<\lambda_2$. Hence, at the point $P_C$, the stable manifold $S$ is tangent to the eigenvector ${v}_1$ associated to $\lambda_1$ which can be written as
\[
{v}_1=(v_{11},1)=\left( \dfrac{\eta + \sqrt{\Delta}}{2\delta \beta (\beta-1)} ,1 \right).
\]
Thus, an equation for the tangent line to $S$ at $P_C=(A,B)$ is given by
\[
y=l(x)=B+m\left(x-A\right),
\]
where $m=1/v_{11}$ can be written as in \eqref{eqSlopem}. Noting that $(\eta + \sqrt{\Delta})(\eta - \sqrt{\Delta})=-4\alpha\beta\delta(\alpha -1)(\beta -1)<0$, we conclude that $\eta + \sqrt{\Delta}>0>\eta - \sqrt{\Delta}$, which implies that $m>0$. Similarly, an  eigenvector ${v}_2$ associated to $\lambda_2$ is given by
\[
{v}_2=(v_{21},1)=\left( \dfrac{\eta - \sqrt{\Delta}}{2\delta \beta (\beta-1)},1 \right).
\]
Thus, the equation for the tangent line to $U$ at $P_C=(A,B)$ is given by
\begin{equation}
y=l_u(x)=B+m_u\left(x-A\right),
\label{eqUTangLine}
\end{equation}
where $m_u=1/v_{21}$ is given by
\[
m_u=-\dfrac{\left(\sqrt{\Delta}+ \eta \right)}{2\alpha (\alpha -1)}
.
\]
Since $\eta + \sqrt{\Delta}>0$, we have $m_u<0$.
\end{proof}

\begin{lemma}
\label{LemaKs}
Assume $\alpha,\beta>1$, $\delta>0$ and consider the open sets \emph{(}rectangles\emph{)}
\[
K_L = (0,A) \times (0,B)
\ \
{\rm and}
\ \
K_R = (A,\infty) \times (B,\infty).
\]
Then:
\begin{description}
\item[1.] $K_L$ and $K_R$ are negatively invariant with respect the flow of $F$;
\item[2.] The components of $F$ satisfy $f(x,y)>0$ and $g(x,y)>0$ for all $(x,y)\in K_L$ and $f(x,y)<0$ and $g(x,y)<0$ for all $(x,y)\in K_R$.
\end{description}
\end{lemma}
\begin{proof}
\begin{description}

\item[1.] Note that, along the horizontal line $y=B$, the vertical component of $F$ is
\[
g(x,B)=\delta\beta B(A-x),
\]
while along the vertical line $x=A$, the horizontal component of $F$ is
\[
f(A,y)=\alpha A(B-y).
\]
These equations imply that, at each segment of the boundaries of $K_L$ and $K_R$ within the lines $x=A$ and $y=B$, the vector field $-F$ points inwards $K_L$ and $K_R$. Furthermore, the left and the bottom side of the boundary of $K_L$ are contained in the $x$ and $y$-axis which are invariant. Therefore, $K_L$ and $K_R$ are positively invariant with respect to the flow of $-F$, i.e., $K_L$ and $K_R$ are negatively invariant with respect to the flow of $F$;

\item[2.] For $(x,y)\in K_L$, i.e., $0<x<A$ and $0<y<B$, we have
\[
x+\alpha y<A+\alpha B =1
,\
y+\beta x <B+\beta A = 1,
\]
which imply that $1-x-\alpha y>0$ and $1-y-\beta x>0$. Thus, $f(x,y)>0$ and $g(x,y)>0$ for $(x,y) \in K_L$. Analogously, the reverse inequalities hold for $(x,y)\in K_R$.
\end{description}
\end{proof}

The following theorem characterizes $S$ as a union of heteroclinic connections (item 1), as a graph of a function $y=s(x)$ (item 2) and as the boundary between the basins of attraction (item 3). See Figure \ref{FigMain} for an illustration.

\begin{theorem}
\label{teoremaS}
Assume $\alpha,\beta>1$, $\delta>0$. Let $\hat{S}=S\cup\{P_0\}$. Then:
\begin{description}
\item[1.] The stable manifold $S$ of $P_C$ is the disjoint union
\begin{equation}
S= S_L \cup \lbrace P_C \rbrace \cup S_R,
\label{eqSetsS}
\end{equation}
where the left branch $S_L$ is a heteroclinic connection from $P_0$ to $P_C$ and the right branch $S_R$ is a heteroclinic connection from a point at the infinity to $P_C$. Moreover, $S_L\subset K_L$ and  $S_R \subset K_R$;
\item[2.] For each $x\geq 0$, there is a unique $y=s(x)$ such that $(x,s(x)) \in \hat{S}$, i.e., $\hat{S}$ is a graph of a function whose domain in the whole nonnegative $x$-axis. The function $s(x)$ is smooth;
\item[3.] $\hat{S}$ is the boundary, or separatrix, between the basins of attraction ${\cal B}_N $ and $ {\cal B}_I$. More specifically, solutions starting below $S$ converge to $P_N$ and solutions starting above $S$ converge to $P_I$, i.e.,
\[
{\cal B}_N = \lbrace (x,y) \in \mathbb{R}^2_+: \ y<s(x) \rbrace
\ \ {\rm and} \ \
{\cal B}_I = \lbrace (x,y) \in \mathbb{R}^2_+: \ y>s(x) \rbrace.
\]
\end{description}
\end{theorem}

\begin{proof}
\begin{description}

\item[1.] From Lemma \ref{LemaReta}, we know that, locally at the point $P_C=(A,B)$, the stable manifold $S$ is tangent to the line $l(x)$ in \eqref{eqTangLine} with slope $m>0$. From the Stable Manifold Theorem, there exists $\epsilon_1>0$ and a smooth function $s(x)$ defined in $[A-\epsilon_1,A)$ such that for $x_1 =A-\epsilon_1$ we have $s(x_1)<B$ and
\begin{equation}
\lbrace \phi (t,x_1,s(x_1)):\ t\geq 0 \rbrace =
S \cap L_s,
\label{eqScapLs}
\end{equation}
where $L_s=[x_1,A) \times [s(x_1),B)\subset K_L$. Define $$S_L = \lbrace \phi (t,x_1,s(x_1));\ t\in \mathbb{R}\rbrace .$$ By the invariance of $S$, and since $(x_1,s(x_1))\in S$, it follows that $S_L \subset S$. From \eqref{eqScapLs} we have that $\phi(t,x_1,s(x_1))\in K_L$ for all $t\geq 0$. Further, since $K_L$ is negatively invariant with respect the flow of $F$ (Lemma \ref{LemaKs}), we conclude that $\phi(t,x_1,s(x_1))\in K_L$ for all $t\in \mathbb{R}$. Therefore, $S_L \subset K_L$.

Now, we prove that $S_L$ is a heteroclinic connection between $P_0$ and $P_C$. Denote the $\alpha$-limit set of $\phi(t,x_1,s(x_1))$ by $L_\alpha$. Then $L_\alpha \in \bar{K_L}$ (the closure of $K_L$). Since \eqref{siscompeticao} has no limit cycles and $\bar{K_L}$ is a compact set, by the Poincar\'e--Bendixson Theorem, it follows that $L_\alpha$ is formed by a union of equilibrium points of \eqref{siscompeticao}. We claim that $L_\alpha=\lbrace P_0 \rbrace$. To show this, we need to show that $P_C\notin L_\alpha$, because $P_N,P_I \notin L_\alpha$. Suppose by contradiction that $P_C \in L_\alpha$. Then, it would exist $t_*<0$ such that $\phi(t_*,x_1,s(x_1)) \in U$, where $U$ is the unstable manifold $U$ of $P_C$. On the other hand, since $U$ is locally tangent to the line $l_u(x)$ in \eqref{eqUTangLine} with slope $m_u < 0$ (Lemma \ref{LemaReta}), by the Stable Manifold Theorem, there is $\epsilon_2>0$ and a smooth function $u(x)$ defined in $[A-\epsilon_2,A)$ such that for $x_2 =A-\epsilon_2$ we have $u(x_2)>B$ and
\[
\lbrace \phi (t,x_2,s(x_2)):\ t\leq 0 \rbrace =
U \cap R_u,
\]
where $R_u=(x_2,A) \times (B,u(x_2))$. This would imply that $\phi(t_*-t,x_1,s(x_1)) \in R_u$ for $t$ sufficiently large, which is a contradiction, since $\phi(t,x_1,s(x_1))\in K_L$ for all $t\in \mathbb{R}$ and $R_u \cap K_L = \emptyset$. Therefore, it follows that the unique option is $L_\alpha = \lbrace P_0 \rbrace$, i.e., the left branch of $S$ is a heteroclinic connection $S_L\subset K_L$ from $P_0$ to $P_C$. 

The remaining part of the statement, about the the right branch $S_R$, is proved in an analogous way. From Lemma \ref{LemaReta}, there exists $\epsilon_3>0$ and a smooth function $s(x)$ defined in $(A,A+\epsilon_3]$ such that for $x_3 =A+\epsilon_3$ we have $s(x_3)>B$ and
\begin{equation}
\lbrace \phi (t,x_3,s(x_3));\ t\geq 0 \rbrace =
S \cap R_s,
\label{eqScapRs}
\end{equation}
where $R_s=(A,x_3] \times (B,s(x_3)] \subset K_R$. Therefore, $\phi(t,x_3,s(x_3))\in K_R$ for all $t\geq 0$. Since $K_R$ is negatively invariant with respect the flow of $F$, the set $$S_R = \lbrace \phi (t,x_3,s(x_3));\ t\in \mathbb{R}\rbrace \subset S$$ is contained in $K_R$. The $\alpha$-limit set of $\phi(t,x_3,s(x_3))$, denoted by $R_\alpha$, is contained in $\bar{K}_R$, which does not contains  limit cycles neither equilibrium points other than $P_C$. Since $P_C \notin R_\alpha$ (analogous argument as before), it follows by the Poincar\'e--Bendixson Theorem that $S_R$ is unbounded and $R_\alpha$ must be a point at the infinity.

We have defined $S_L$ and $S_R$ and showed that one inclusion in \eqref{eqSetsS} is satisfied. To show the equality, let $(x,y)\in S$. Then, $\phi(t,x,y) \to P_C$ when $t\to \infty$. Therefore, for sufficiently large $t$, either $\phi(t,x,y)\in L_s$ or $\phi(t,x,y)\in R_s$. Since $\phi(t,x,y) \in S$, we have either $\phi(t,x,y)\in S_L$ or $\phi(t,x,y)\in S_R$, which imply, that $(x,y)\in S_L \cup S_R$;

\item[2.] For $x=0$, since the set $\lbrace (0,y);\ y>0\rbrace $ is contained in the basin of attraction of $P_I$, the unique $y$ such that $(0,y)\in \hat{S}$ is $y=0$. Therefore, we define $s(0)=0$.

Now, fix $x$ in the interval $0<x<A$. First, since $S_L$ is a smooth curve connecting $P_C=(A,B)$ to $P_0=(0,0)$, it follows from the Intermediary Value Theorem that there is $y \in (0,B)$ such that $(x,y) \in S_L \subset S$. Suppose by contradiction that there exist $y_2\neq y$, $y_2\geq 0$ such that $(x,y_2)$ is also a point in $S_L$. Consider the trajectory $\phi_L(t)=\phi(t,x,y_1) \in S_L$. We have $\phi_L(0)=(x,y_1)$ and there exists $t_2$ such that $\phi_L(t_2)=(x,y_2)$. Assume $t_2>0$ without loss of generality. It follows from the Mean Value Theorem in $\mathbb{R}^n$ that there is $\bar{t} \in (0,t_2)$ such that
\[
\phi_L'(\bar{t}) = \dfrac{\phi_L(t_2)-\phi_L(0)}{t_2} = \dfrac{(0,y_2-y)}{t_2}.
\]
On the other hand, we have $\phi_L'(\bar{t}) = F(\bar{x},\bar{y}) = (f(\bar{x},\bar{y}),g(\bar{x},\bar{y}))$, where $(\bar{x},\bar{y})=\phi_L(\bar{t}) \in S_L$. This would imply that $f(\bar{x},\bar{y}) =0$, which is a contradiction, since $(\bar{x},\bar{y}) \in S_L \subset K_L$ and $f(x,y)>0$ in $K_L$. Therefore, there is an unique $y>0$ such that $(x,y) \in S_L$.

For $x=A$, we define $y=B$ (from item 1 in this Theorem \ref{teoremaS}, $S$ intersects the vertical line $x=A$ only at $y=B$). For $x>A$, the rationale is analogous to the case $0<x<A$. The only remark is that we have to guarantee that the domain of $s(x)$ is the whole set of values $x>0$, i.e., the right branch of $S$ does not diverges to infinity at a finite $x$. This is guaranteed by the fact that $F$ does not have any singularity neither any equilibrium point in $K_R$. Therefore, $s(x)$ is a well-defined function whose domain is $x\geq 0$. It follows from the Stable Manifold Theorem that $s(x)$ is smooth;

\item[3.] Consider a point $(x,y) \in \mathbb{R}^2_+$ below $\hat{S}$, i.e., such that $y<s(x)$. The only possibilities are either $(x,y) \in {\cal B}(P_N)$ or $(x,y) \in {\cal B}(P_I)$. Since $P_I$ is above $S$, i.e., $1>s(0)$, if $(x,y) \in {\cal B}(P_I)$, then at some point the solution $\phi(t,x,y)$ should satisfy $y=s(x)$, which would imply in $(x,y)\in S$, a contradiction.
\end{description}
\end{proof}

Now, we prove Theorem \ref{teoMain}, item 1. Statements (a) and (b) follow from Theorem \ref{teoremaS}. To prove statement (c), i.e., that $s(x)$ is strictly increasing, note that $s'(A)=m>0$ (Lemma \ref{LemaReta}) and $s'(x)=dy/dx=(dy/dt)/(dx/dt)=g(x,y)/f(x,y)$ for $x \neq A$. Then, it is enough to observe that $( S \setminus \{ (A,B) \} ) \subset K_L \cup K_R$ by Theorem \ref{teoremaS},  item 1, and that $g(x,y)/f(x,y) >0 $ in $K_L \cup K_R$ by Lemma \ref{LemaKs}, item 2.

\subsection{Proof of Theorem \ref{teoMain}, item 2}

In the following, we show that the function $y=s(x)$, describing the separatrix, satisfies the integral equation \eqref{eqIntegralS}. From system \eqref{siscompeticao} we have that the solutions satisfy
$$
\dfrac{d}{dt}\left(
\log y(t)
\right)
=\dfrac{y'(t)}{y(t)}=\delta ( 1 - y(t)- \beta x(t) ), \
\dfrac{d}{dt}\left(
\log x(t)
\right)
= \dfrac{x'(t)}{x(t)}=\delta ( 1 - x(t)- \alpha y(t) ).
$$
Thus,
$$
\dfrac{d}{dt}\left(
\log \dfrac{y(t)}{x(t)^\delta}
\right)
=
\dfrac{d}{dt}\left(
\log y(t) - \delta \log x(t)
\right)
=
\delta  \left(  (\alpha-1) y(t) -(\beta -1) x(t) \right).
$$
Integrating both sides from $0$ to $t$ we obtain that all solutions of system \eqref{siscompeticao} satisfy
$$
\log \dfrac{y(t)}{ x(t)^\delta} - \log \dfrac{y_0}{ x_0^\delta} = 
 \delta \int_0^t  (\alpha-1) y(\tau) -(\beta -1) x(\tau) d\tau,
$$
where $(x_0,y_0)=(x(0),y(0))$. If we consider a solution $(x(t),y(t))$ starting in $S$, we have that $(x(t),y(t)) \to P_C= (A,B)$ when $t \to \infty$. Thus, taking this limit in the last equation, we obtain that the points $(x_0,y_0)$ in $S$ satisfy
$$
\log \dfrac{y_0}{ x_0^\delta} = \log \dfrac{B}{ A^\delta} - 
 \delta \int_0^\infty  (\alpha-1) y(\tau) -(\beta -1) x(\tau) d\tau.
$$
Lemma \ref{LemaInt} below guarantees that the above improper integral converges. Applying the exponential function leads to
$$
y_0=B\left(\dfrac{x_0}{A}\right)^\delta \exp \left ( -\delta \int_0^\infty  (\alpha-1) y(\tau) -(\beta -1) x(\tau) d\tau \right).
$$
Using the following change of variables in the above integral, we rewrite it terms of $x$: if $x=x(\tau)$, then $x(0)=x_0$, $x(\infty)=A$ and $dx=x'(\tau)d\tau=f(x(\tau),y(\tau))d\tau$. Further, since $(x(\tau),y(\tau))\in S$, we have $y(\tau)=s(x(\tau))$, and thus $d\tau = dx/f(x,s(x))$. Therefore, the last equation assumes the form
$$
y_0=B\left(\dfrac{x_0}{A}\right)^\delta \exp \left ( -\delta \int_{x_0}^A  \dfrac{(\alpha-1) s(x) -(\beta -1) x}{f(x,s(x))} dx \right).
$$
Replacing $x=\bar{x} $ inside the integral and $x_0=x,$ $y_0=y=s(x)$, we conclude that points $(x,s(x))\in S$ satisfy the integral equation
$$
s(x)=B\left(\dfrac{x}{A}\right)^\delta \exp \left ( \delta \int_A^x  \dfrac{(\alpha-1) s(\bar{x}) -(\beta -1) \bar{x}}{f(\bar{x},s(\bar{x}))} d\bar{x} \right),
$$
which proofs the first statement of item 2. The last equation is equivalent to
$$
s(x)=s^*(x) C(x,s(x)),
$$
where $s^*(x)=B(x/A)^\delta$ and
$$
C(x,s(x)) = 
\exp \left ( \delta \int_A^x  \dfrac{(\alpha-1) s(\bar{x}) -(\beta -1) \bar{x}}{f(\bar{x},s(\bar{x}))} d\bar{x} \right).
$$
Note that, for $\delta=1$, $s{^\ast}(x)=B x/A$ satisfies the equation $g(x,s^{\ast}(x))=(s^{\ast})'(x)f(s,s^{\ast}(x))$, for all $x\geq 0$. Thus, $y(t)=s^*(x(t))$ is a solution of system \eqref{siscompeticao} when $\delta=1$. Hence, $s(x)$ and $s^*(x)$ coincide when $\delta=1$. By continuity arguments, $s^*(x)$ is a uniform approximation for $s(x)$ for $|\delta-1|$ small enough.

\begin{lemma} Let $(x(t),y(t))$ be a solution for system \eqref{siscompeticao}, with $(x(0),y(0)) \in {S}$. Then, the improper integral 
\[
\displaystyle \int_0^\infty  (\alpha-1) y(\tau) -(\beta -1) x(\tau) d\tau
\]
converges.
\label{LemaInt}
\end{lemma}
\begin{proof}
Since ${S}$ in invariant and is the stable manifold of $P_C=(A,B)$, we have $(x(t),y(t)) \in {S}$. From the Hartman--Grobman Theorem, there exists $t_1>0$ and continuous bounded functions $k_1(t),k_2(t)$ such that 
\begin{equation}
(x(t),y(t)) = (A,B)+ e^{\lambda_1 t} v_1+e^{\lambda_1 t} (k_1(t),k_2(t))
\label{eqXYinS}
\end{equation}
for $t\geq t_1$, with $\lim_{t\to \infty}{k_i(t)}=0 $, where $\lambda_1<0$ is the eigenvalue of the Jacobian matrix at $(A,B)$ and $v_1=(v_{11},1)$ is the associated eigenvector (see Lemma \ref{LemaReta}). Thus, for $\tau > t_1$, using \eqref{eqXYinS} we obtain that the integrand is written as
$$
(\alpha-1) y(\tau) -(\beta -1) x(\tau) = (\alpha\beta-1) e^{\lambda_1 t} \left(A +A k_2(t)-Bv_{11}-B k_1(t) \right).
$$
Since $\lambda_1<0$ and $k_1(t)$ and $k_2(t)$ are continuous and bounded, it follows that the integral
\[
\lim_{t\to\infty} \int_{t_1}^t (\alpha-1) y(\tau) -(\beta -1) x(\tau)  d\tau
\]
converges, which shows the lemma.
\end{proof}
 
\subsection{Proof of Theorem \ref{teoMain}, item 3}

A solution $(x(t),y(t))$ of system \eqref{siscompeticao} starting in ${S}$ satisfies $y(t)=s(x(t))$. Differentiating both sides with respect to $t$, we obtain $g(x(t),y(t))=s'(x(t)) f(x(t),y(t))$. Therefore, $s'(x)=g(x,s(x))/f(x,s(x))$ for $x>0$ and $x\neq A$. On the other hand, we know that $s'(A)$ is given by the slope of the eigenvector $v_1$, i.e., $s'(A)=m$, as given in \eqref{eqSlopem}. Since $s(x)$ is smooth by Theorem \ref{teoremaS}, it follows that the function
\[
s'(x)=
\left\lbrace
\begin{array}{cl}
\dfrac{g(x,s(x))}{f(x,s(x))}, & 0< x \neq A \vspace{0.2cm}\\
m, & x=A,
\end{array}
\right.
\]
is smooth. Thus, the function
\[
h(x,y,\xi)= \left \{
\begin{array}{cl}
\dfrac{g(x,y)}{f(x,y)}, & (x,y) \neq (A,B)\vspace{0.2cm}\\
m, & (x,y)=(A,B),
\end{array}
\right.
\]
is smooth, where $\xi \in \{\alpha,\beta,\delta\}$ is one chosen model parameter. We denote by $\varphi(x,x_0,\xi)$ the maximal solution of the differential equation
\[
y'=h(x,y,\xi),
\]
through the point $x_0>0$, which exists since $h$ is smooth.

We start by presenting the following lemma, which provides a formula to calculate the derivative of $y=s(x)=s(x,\xi)$ with respect to the parameter $\xi \in \{\alpha,\beta,\delta\}$.

\begin{lemma}\label{LemaTI3}
$S$ is differentiable with respect to the model parameters, i.e., the function $y=s(x,\xi)$, $x>0$, is differentiable and its derivative with respect to the parameter $\xi \in \{ \alpha, \beta , \delta \}$ is given by
\[
\dfrac{\partial}{\partial \xi}s(x,\xi) =\exp\left(\int_{A}^x a(s,\xi)\,\mathrm{d}s\right)\left(C+\int_{A}^x b(u,\xi)\exp\left(-\int_{A}^u a(s,\xi)\,\mathrm{d}s\right)\,\mathrm{d}u \right),
\]
where
\[
{\setlength\arraycolsep{2pt}
\begin{array}{rcl}
C & = & \dfrac{\partial }{\partial \xi}(B)-m\dfrac{\partial }{\partial \xi}(A),\vspace{0.2cm}\\
a(x,\xi) & = & \dfrac{\partial }{\partial y} h(x,\varphi(x,A,\xi),\xi),\vspace{0.2cm}\\
b(x,\xi) & = & \dfrac{\partial }{\partial \xi} h(x,\varphi(x,A,\xi),\xi).
\end{array}}
\]
\end{lemma}
\begin{proof} We know that $s(x,\xi)=\varphi(x,A,\xi)$ and, therefore,
\[
\dfrac{\partial}{\partial \xi}s(x,\xi)=\dfrac{\partial}{\partial \xi}\varphi(x,A,\xi),
\]
where $\varphi(x,A,\xi)$ satisfies
\begin{equation}\label{eq:PVI}
\left \{
{\setlength\arraycolsep{2pt}
\begin{array}{rcl}
\dfrac{\partial}{\partial x}\varphi(x,A,\xi) & = & h(x,\varphi(x,A,\xi),\xi),\vspace{0.2cm}\\
\varphi(A,A,\xi) & = & B.
\end{array}}
\right.
\end{equation}

Differentiating both sides of \eqref{eq:PVI}, with respect to $\xi \in \{\alpha,\beta,\delta\}$, it follows that
\begin{equation*}
\left \{
{\setlength\arraycolsep{2pt}
\begin{array}{rcl}
\dfrac{\partial}{\partial x}\left(\dfrac{\partial}{\partial \xi}\varphi(x,A,\xi)\right) & = & \dfrac{\partial }{\partial y}h(x,\varphi(x,A,\xi),\xi)\dfrac{\partial}{\partial \xi}\varphi(x,A,\xi)+\dfrac{\partial }{\partial \xi}h(x,\varphi(x,A,\xi),\xi),\vspace{0.2cm}\\
\dfrac{\partial}{\partial \xi}\varphi(A,A,\xi) & = & \dfrac{\partial }{\partial \xi}(B)-h(A,B,\xi)\dfrac{\partial }{\partial \xi}(A).
\end{array}}
\right.
\end{equation*}
The above initial value problem is a linear problem with the form
\[
\dfrac{\partial }{\partial x} z(x,\xi) = a(x,\xi) z(x,\xi)+b(x,\xi),  \ \ z(a,\xi)=C.
\]
for $z=\partial \varphi(x,A,\xi)/\partial \xi$. Solving it with standard methods we obtain the desired result.
\end{proof}

Now we take $\xi=\alpha$ in Lemma \ref{LemaTI3} and conclude that $\partial s(x,\alpha)/\partial \alpha<0$, for each $x>0$ and $\alpha>1$. In fact, by \eqref{eqDerivAparam} and \eqref{eqSlopem}, it follows that $C<0$ and a straightforward calculation shows that
\[
{\setlength\arraycolsep{2pt}
\begin{array}{rcl}
b(x,\alpha) & = & \dfrac{\partial }{\partial \alpha} h(x,s(x,\alpha),\alpha)\vspace{0.2cm}\\
& = & \dfrac{\delta s(x,\alpha)^2(1-\beta x-s(x,\alpha))}{x(1-x-\alpha s(x,\alpha))},\quad \forall x>0.
\end{array}}
\]
So, $b(x,\alpha)>0$, if $0<x<A$ and $b(x,\alpha)<0$, if $x>A$. 

In the same way, if $\xi=\beta$, then $\partial s(x,\beta)/\partial \beta>0$, for each $x>0$ and $\beta>1$, since $C>0$, by  \eqref{eqDerivBparam} and \eqref{eqSlopem}, and
\[
{\setlength\arraycolsep{2pt}
\begin{array}{rcl}
b(x,\beta) & = & \dfrac{\partial }{\partial \beta} h(x,s(x,\beta),\beta)\vspace{0.2cm}\\
& = & -\dfrac{\delta s(x,\beta)}{1-x-\alpha s(x,\beta)},\quad \forall x>0.
\end{array}}
\]
Thus, $b(x,\beta)<0$, if $0<x<A$ and $b(x,\beta)>0$, if $x>A$.

Finally, when $\xi=\delta$, we have $C=0$ and
\[
{\setlength\arraycolsep{2pt}
\begin{array}{rcl}
b(x,\delta) & = & \dfrac{\partial }{\partial \delta} h(x,s(x,\delta),\delta)\vspace{0.2cm}\\
& = & \dfrac{s(x,\delta)(1-\beta x-s(x,\delta))}{x(1-x-\alpha s(x,\delta))}>0,\quad \forall x>0.
\end{array}}
\]
In this case, we conclude that $\partial s(x,\delta)/\partial \delta<0$, if $0<x<A$ and $\delta>0$. On the other hand, $\partial s(x,\delta)/\partial \delta>0$, when $x>A$ and $\delta>0$.

\subsection{Proof of Theorem \ref{teoMain}, item 4}

Changing the time variable $t \mapsto \delta t$, system \eqref{siscompeticao} becomes

\begin{equation}\label{sisPert7}
\left\{
{\setlength\arraycolsep{2pt}
\begin{array}{rcl}
\delta \dfrac{dx}{dt}  & = &  x \left(1-x-\alpha y\right), \vspace{0.2cm} \\
\dfrac{dy}{dt}   & = & y   \left(1-y-\beta x\right).
\end{array}}
\right.
\end{equation}
Since $\delta \to 0$, we use Singular Perturbation Theory. The solutions with $x(0)>0$ and $y(0)>0$ will quickly approach the slow manifold given by the line passing by $P_N$,
\[
1-x-\alpha y=0.
\]
On the slow manifold, the dynamics is given by
\[
\dfrac{dy}{dt}   = y   \left(1-y-\beta (1-\alpha y) \right).
\]
It is easy to see that $dy/dt>0$ for $y>B$, and $dy/dt<0$ for $y<B$. Thus, solutions above the line $y=B$ converge to $P_I$, while solutions below this line converge to $P_N$. Therefore, the separatrix is given by this line. This proves statement (a). 

The proof of statement (b) is analogous. System \eqref{siscompeticao} can be written as
\begin{equation}\label{sisPert8}
\left\{
{\setlength\arraycolsep{2pt}
\begin{array}{rcl}
\dfrac{dx}{dt}  & = &  x \left(1-x-\alpha y\right), \vspace{0.2cm} \\
\dfrac{1}{\delta}\; \dfrac{dy}{dt}   & = & y   \left(1-y-\beta x\right).
\end{array}}
\right.
\end{equation}
Using Singular Perturbation Theory, we find that the solutions with $x(0)>0$ and $y(0)>0$ will quickly approach the slow manifold, which is the line passing by $P_I$,
\[
1-y-\beta x=0.
\]
Along the slow manifold, the dynamics is given by
\[
\dfrac{dx}{dt}   = x   \left(1-x-\alpha (1-\beta x) \right).
\]
For $x>A$ we have $dx/dt>0$, while $dx/dt<0$ for $x<A$. Therefore, the line $x=A$ is the separatrix.

\section{Conclusion}\label{secConclusion}

Understanding and quantifying resilience is crucial to anticipate the possible critical transitions occurring in a system. Characterizing the influence of the system intrinsic parameters on resilience is the first step towards the direction to employ rational decisions and acting on the system in order to avoid undesirable transitions. In this paper, we apply several tools of qualitative theory of ordinary differential equations to quantify and characterize the resilience of competing populations as described by the classical Lotka-Volterra model.

We study such model in the bi-stable scenario, in the context of high interspecific competition where one population is regarded as native and other as invader. The mathematical methods used here allowed us to analytically characterize an invariant manifold of a saddle-point as the graph of a smooth function and to characterize its behavior in terms of the model parameters. Since such manifold is the boundary between the basins of attraction of the equilibrium points corresponding to survival of one population and extinction to the other, those characterizations allowed to deduce the qualitative behavior of two resilience measures of one of the populations (regarded as the ``native'' population). Such measures are the \textit{precariousness}, which quantifies the minimum number of invader individuals required to lead to a successful invasion (which completely extincts the native populations) and the \textit{latitude}, which quantifies the probability that a random but natural initial condition leads to the survival of the native population.

Our results showed how adaptations on the native population behavior may protect it against the invasion of a second population. As expected, the results indicate that increasing competitiveness is always a advantageous strategy for a population, while with respect to reproduction, it may not be the case. We showed that within a \textit{pioneering} context where both populations initiate with few individuals within the environment, increasing reproduction leads to an increase in resilience, thereby increasing the chances of survival. On the contrary, within an environment initially dominated by a native population which is invaded by a second population, decreasing reproduction leads to an increase in resilience, thereby increasing the chances of survival. All results devised here are analytic and extend the numerical results that we obtained in a previous publication \cite{fassoni2014basins}.

In summary, we believe that this work brings near to each other the theoretical concepts of ecological resilience and the mathematical methods of differential equations and has the potential to stimulate the development and application of new mathematical tools for ecological resilience. As far as we now, no previous work has developed and applied analytical methods to study ecological resilience with the approach provided here. Furthermore, as a outcome of such resilience analysis for the model considered here, the biological implications provide interesting insights for the dynamics of competing species and show the concrete benefits and importance of considering ecological resilience when analyzing mathematical models in biology.

\bigskip

\noindent {\bf Acknowledgement}: The first author was partially supported by CAPES/P\'os-Doutorado no Exterior [grant number 88881.119037/2016-01]. The second author is partially supported by Funda\c c\~ao de Amparo \`a Pesquisa do Estado de Minas Gerais -– FAPEMIG [project number APQ–01158–17].



\begin{thebibliography}{99}

\bibitem{walker2004resilience}
Brian Walker, Crawford~S Holling, Stephen~R Carpenter, and Ann Kinzig.
\newblock Resilience, adaptability and transformability in social--ecological
  systems.
\newblock {\em Ecology and Society}, 9(2):5, 2004.

\bibitem{folke2006resilience}
Carl Folke.
\newblock Resilience: The emergence of a perspective for social--ecological
  systems analyses.
\newblock {\em Global Environmental Change}, 16(3):253--267, 2006.

\bibitem{scheffer2001catastrophic}
Marten Scheffer, Steve Carpenter, Jonathan~A Foley, Carl Folke, and Brian
  Walker.
\newblock Catastrophic shifts in ecosystems.
\newblock {\em Nature}, 413(6856):591--596, 2001.

\bibitem{barnosky2012approaching}
Anthony~D Barnosky, Elizabeth~A Hadly, Jordi Bascompte, Eric~L Berlow, James~H
  Brown, Mikael Fortelius, Wayne~M Getz, John Harte, Alan Hastings, Pablo~A
  Marquet, et~al.
\newblock Approaching a state shift in earth’s biosphere.
\newblock {\em Nature}, 486(7401):52, 2012.

\bibitem{trotta2012global}
Laura Trotta, Eric Bullinger, and Rodolphe Sepulchre.
\newblock Global analysis of dynamical decision-making models through local
  computation around the hidden saddle.
\newblock {\em PLoS One}, 7(3):e33110, 2012.

\bibitem{scheffer2013migraine}
Marten Scheffer, Albert van~den Berg, and Michel~D Ferrari.
\newblock Migraine strikes as neuronal excitability reaches a tipping point.
\newblock {\em PloS one}, 8(8):e72514, 2013.

\bibitem{fassoni2014mathematical}
Artur~C Fassoni and Marcelo~L Martins.
\newblock Mathematical analysis of a model for plant invasion mediated by
  allelopathy.
\newblock {\em Ecological complexity}, 18:49--58, 2014.

\bibitem{fassoni2014basins}
Artur~C Fassoni, Lucy~T Takahashi, and La\'ercio~J dos Santos.
\newblock Basins of attraction of the classic model of competition between two
  populations.
\newblock {\em Ecological Complexity}, 18:39--48, 2014.

\bibitem{konstorum2016feedback}
Anna Konstorum, Thomas Hillen, and John Lowengrub.
\newblock Feedback regulation in a cancer stem cell model can cause an allee
  effect.
\newblock {\em Bulletin of Mathematical Biology}, 78(4):754--785, 2016.

\bibitem{fassoni2017ecological}
Artur~C Fassoni and Hyun~M Yang.
\newblock An ecological resilience perspective on cancer: insights from a toy
  model.
\newblock {\em Ecological Complexity}, 30:34--46, 2017.

\bibitem{carpenter2001metaphor}
Steve Carpenter, Brian Walker, J~Marty Anderies, and Nick Abel.
\newblock From metaphor to measurement: resilience of what to what?
\newblock {\em Ecosystems}, 4(8):765--781, 2001.

\bibitem{mitra2015integrative}
Chiranjit Mitra, J{\"u}rgen Kurths, and Reik~V Donner.
\newblock An integrative quantifier of multistability in complex systems based
  on ecological resilience.
\newblock {\em Scientific Reports}, 5:16196, 2015.

\bibitem{meyer2016mathematical}
Katherine Meyer.
\newblock A mathematical review of resilience in ecology.
\newblock {\em Natural Resource Modeling}, 29(3):339--352, 2016.

\bibitem{menck2013basin}
Peter~J Menck, Jobst Heitzig, Norbert Marwan, and J{\"u}rgen Kurths.
\newblock How basin stability complements the linear-stability paradigm.
\newblock {\em Nature Physics}, 9(2):89--92, 2013.

\bibitem{holling1973resilience}
Crawford~S Holling.
\newblock Resilience and stability of ecological systems.
\newblock {\em Annual Review of Ecology and Systematics}, pages 1--23, 1973.

\bibitem{chiang1988stability}
Hsiao-Dong Chiang, Morris~W Hirsch, and Felix~F Wu.
\newblock Stability regions of nonlinear autonomous dynamical systems.
\newblock {\em Automatic Control, IEEE Transactions on}, 33(1):16--27, 1988.

\bibitem{genesio1985estimation}
Roberto Genesio, Michele Tartaglia, and Antonio Vicino.
\newblock On the estimation of asymptotic stability regions: State of the art
  and new proposals.
\newblock {\em IEEE Transactions on automatic control}, 30(8):747--755, 1985.

\bibitem{krauskopf2007computing}
Bernd Krauskopf and Hinke~M Osinga.
\newblock Computing invariant manifolds via the continuation of orbit segments.
\newblock In {\em Numerical Continuation Methods for Dynamical Systems}, pages
  117--154. Springer, 2007.

\bibitem{guckenheimer2004fast}
John Guckenheimer and Alexander Vladimirsky.
\newblock A fast method for approximating invariant manifolds.
\newblock {\em SIAM Journal on Applied Dynamical Systems}, 3(3):232--260, 2004.

\bibitem{cavoretto2016robust}
Roberto Cavoretto, Alessandra De~Rossi, Emma Perracchione, and Ezio Venturino.
\newblock Robust approximation algorithms for the detection of attraction
  basins in dynamical systems.
\newblock {\em Journal of Scientific Computing}, 68(1):395--415, 2016.

\bibitem{delboni2017mathematical}
Roberta~Regina Delboni and Hyun~Mo Yang.
\newblock Mathematical model of interaction between bacteriocin-producing
  lactic acid bacteria and listeria. part 2: Bifurcations and applications.
\newblock {\em Bulletin of mathematical biology}, 79(10):2273--2301, 2017.

\bibitem{meyer2018quantifying}
Katherine Meyer, Alanna Hoyer-Leitzel, Sarah Iams, Ian Klasky, Victoria Lee,
  Stephen Ligtenberg, Erika Bussmann, and Mary~Lou Zeeman.
\newblock Quantifying resilience to recurrent ecosystem disturbances using
  flow--kick dynamics.
\newblock {\em Nature Sustainability}, 1(11):671, 2018.

\bibitem{ledzewicz2015dynamics}
Urszula Ledzewicz, Beiirooz Amini, and Heinz Sch{\"a}ttler.
\newblock Dynamics and control of a mathematical model for metronomic
  chemotherapy.
\newblock {\em Mathematical Biosciences and Engineering: MBE},
  12(6):1257--1275, 2015.

\bibitem{murray2007mathematical}
James~D Murray.
\newblock Mathematical biology: I. an introduction (interdisciplinary applied
  mathematics)(pt. 1), 2007.

\bibitem{chiralt2017quantitative}
Cristina Chiralt, Antoni Ferragut, Armengol Gasull, and Pura Vindel.
\newblock Quantitative analysis of competition models.
\newblock {\em Nonlinear Analysis: Real World Applications}, 38:327--347, 2017.

\bibitem{hofbauer1998evolutionary}
Josef Hofbauer and Karl Sigmund.
\newblock {\em Evolutionary games and population dynamics}.
\newblock Cambridge University Press, 1998.

\end{thebibliography}
\end{document}